\documentclass[hidelinks,onefignum,onetabnum]{siamart251216}
\usepackage{amssymb,mathrsfs}
\usepackage{mathtools}
\usepackage{algpseudocode}
\usepackage{cite}
\newsiamremark{remark}{Remark}
\newcommand{\R}{\mathbb{R}}
\newcommand{\norm}[1]{\left\lVert #1\right\rVert}
\newcommand{\spec}{\operatorname{spec}}
\newcommand{\diag}{\operatorname{diag}}
\newcommand{\Lip}{\operatorname{Lip}}
\newcommand{\Id}{\mathbb{I}}
\newcommand{\prox}{\operatorname{prox}}
\newcommand{\dist}{\operatorname{dist}}
\newcommand{\classF}{\mathscr{F}}
\newcommand{\DAI}{\mathrm{DAI}}
\newcommand{\M}{\mathcal{M}}
\newcommand{\I}{\mathcal{I}}
\newcommand{\fc}{\mathfrak{c}}
\newcommand{\fs}{\mathfrak{s}}
\DeclarePairedDelimiter{\abs}{\lvert}{\rvert}

\DeclareRobustCommand{\ADA}{%
  \textup{%
    A%
    \kern-.07em%
    \lower.21ex\hbox{D}%
    \kern-.09em%
    A%
  }%
}

\headers{Descent--Ascent Structure of Uniform Nondegeneracy}{Hua Su, Lei Zhang, and Jin Zhao}
\hypersetup{
  pdftitle={Descent--Ascent Structure of Uniform Nondegeneracy},
  pdfauthor={Hua Su, Lei Zhang, and Jin Zhao}
}

\title{\small Descent--Ascent Structure of Uniform Nondegeneracy
\thanks{Submitted to the editors DATE.
\funding{L.Z. was supported by the National Natural Science Foundation of China (No. 12225102, T2321001, and 12288101). J.Z. was supported in part by the Beijing Natural Science Foundation (Grant No. JR25003), the National Natural Science Foundation of China (Grant No.12671504), and the Beijing Outstanding Young Scientist Program (No.JWZQ20240101027).}}
}
\author{
Hua Su\thanks{\raggedright Beijing International Center for Mathematical Research, Peking University, Beijing, China. E-mail: \email{suhua@pku.edu.cn}.}
\and
Lei Zhang\thanks{\raggedright Corresponding author. School of Mathematical Sciences, Beijing International Center for Mathematical Research, Center for Quantitative Biology, Center for Machine Learning Research, Institute for Artificial Intelligence, Peking University, Beijing, China. E-mail: \email{zhangl@math.pku.edu.cn}.}
\and
Jin Zhao\thanks{\raggedright Academy for Multidisciplinary Studies, Capital Normal University, Beijing, China. E-mail: \email{zjin@cnu.edu.cn}.}
}

\begin{document}
\maketitle

\begin{abstract}
Classical models such as smooth strongly convex minimization and strongly convex--strongly concave optimization combine two types of curvature information: curvature remains uniformly separated from zero, and its positive and negative orientations are prescribed in advance. Uniform nondegeneracy disentangles these roles, retaining the separation from zero while allowing the curvature orientation to vary across the domain.
We identify a family of descent--ascent inequalities (DAI) that characterizes uniform nondegeneracy and naturally gives rise to Alternating Descent Ascent (\ADA{}), a simple explicit method that alternates full-gradient descent and ascent. \ADA{} converges globally and linearly to the unique stationary point from arbitrary initialization, with dimension-free first-order complexity. The framework recovers the smooth strongly convex setting as a limit and encompasses the strongly convex--strongly concave and linearly constrained settings, highlighting uniform nondegeneracy as a common first-order structure across various classical regimes.
\end{abstract}

\begin{keywords}
uniform nondegeneracy, descent--ascent inequalities, alternating descent ascent, first-order methods, global convergence
\end{keywords}

\begin{MSCcodes}
90C26, 90C60, 65K05
\end{MSCcodes}

\section{Introduction}

Much of classical first-order optimization is built around models with uniform curvature control: strong convexity keeps curvature uniformly positive~\cite{nesterov2018lectures}, strong concavity keeps it uniformly negative, while strongly convex--strongly concave (SCSC) structure prescribes convex and concave orientations through a fixed splitting of the variables~\cite{rockafellar1976monotone}. These models intertwine two distinct forms of structure: quantitative separation of curvature from zero and prescribed information about how its positive and negative directions are organized. In our preceding work~\cite{su2026uniform}, we introduced uniform nondegeneracy to disentangle these roles, requiring curvature to remain quantitatively separated from zero while allowing its orientation to vary across the domain.

For $E\in C^2(\R^d)$ with gradient $g:=\nabla E$ and Hessian $\nabla^2 E$, uniform nondegeneracy with parameters $L_+\ge\mu_+>0$ and $L_-\ge\mu_->0$ amounts to the spectral condition
\begin{equation}\label{eq:spectrum}
\spec\left(\nabla^2 E(x)\right)\subseteq\mathcal{S}:=[-L_-,-\mu_-]\cup[\mu_+,L_+],
\qquad \forall x\in\R^d,
\end{equation}
which extends to $C^1$ objectives through the equivalent pairwise secant formulation:
\begin{subequations}\label{eq:secants}
\begin{align}
\norm{\Delta g-c_\mu\Delta x}&\ge\bar\mu\norm{\Delta x},\label{eq:secant-inner}\\
\norm{\Delta g-c_L\Delta x}&\le\bar L\norm{\Delta x},\label{eq:secant-outer}
\end{align}
\end{subequations}
for every $x_1,x_2\in\R^d$, with $\Delta x:=x_1-x_2$, $\Delta g:=g(x_1)-g(x_2)$, and
\begin{equation*}
c_\mu:=\frac{\mu_+-\mu_-}{2},\quad
\bar\mu:=\frac{\mu_++\mu_-}{2},\quad
c_L:=\frac{L_+-L_-}{2},\quad
\bar L:=\frac{L_++L_-}{2}.
\end{equation*}

We denote the resulting $C^1$ function class by $\classF_{\mu_\pm,L_\pm}(\R^d)$. Within this geometry, prescribed curvature orientations recover the familiar smooth strongly convex, strongly concave, and SCSC settings, while the full class also admits genuinely indefinite objectives with varying curvature orientations and no fixed convex--concave splitting.
Moreover, uniform nondegeneracy by itself already determines strong global structure~\cite[Sections~2--3]{su2026uniform}: the gradient map is a global bi-Lipschitz homeomorphism, and hence $E$ has a unique stationary point $x^\star$, for which we established global first-order convergence from arbitrary initialization together with dimension-free complexity guarantees. Taken together, these results establish uniform nondegeneracy as an independent organizing principle for global first-order optimization.

\subsection{Main Results}

We reveal a descent--ascent structure underlying uniform nondegeneracy, with a particularly simple explicit algorithmic realization:

\begin{algorithm}[H]
\caption{Alternating Descent Ascent (\ADA{})}\label{alg:ada}
\begin{algorithmic}[1]
\State Initialize $\sigma\in\{\downarrow,\uparrow\}$
\While{$\norm{g(x)}>\varepsilon$}
\If{$\sigma=\;\downarrow$}
\State $x:=x-\tau_+g(x)$ \Comment{Gradient descent with step size $\tau_+>0$}
\State $\sigma:=\;\uparrow$
\Else
\State $x:=x+\tau_-g(x)$ \Comment{Gradient ascent with step size $\tau_->0$}
\State $\sigma:=\;\downarrow$
\EndIf
\EndWhile
\end{algorithmic}
\end{algorithm}

Underlying this scheme is a deeper structure of uniform nondegeneracy, formalized through a family of descent--ascent inequalities (DAI). Using the standard Bregman difference $D_f(u,v):=f(u)-f(v)-\langle\nabla f(v),u-v\rangle$, we define DAI as follows.

\begin{definition}[Descent--ascent inequality]
We say that $E\in C^1(\R^d)$ satisfies the descent--ascent inequality (DAI) with weights $\rho_\pm\ge0$ and scales $\Lambda_\pm>0$, abbreviated as $\DAI(\rho_\pm;\Lambda_\pm)$, if, for every $x,y,z\in\R^d$,
\begin{equation}\label{eq:dai}
\rho_+\Phi_+(x,y)+\rho_-\Phi_-(x,z)
\ge \left\langle\nabla_x\Phi_+(x,y),\nabla_x\Phi_-(x,z)\right\rangle,
\end{equation}
where
\begin{equation*}
\Phi_+(x,y):=\frac12\norm{x-y}^2-\frac{1}{\Lambda_+}D_E(x,y),\quad
\Phi_-(x,z):=\frac12\norm{x-z}^2+\frac{1}{\Lambda_-}D_E(x,z).
\end{equation*}
\end{definition}

The terminology comes from the following algebraically equivalent form: writing the descent and ascent maps in \Cref{alg:ada} at step sizes $\tau_\pm=1/\Lambda_\pm$ as
\begin{equation*}
x^\downarrow:=x-\frac{g(x)}{\Lambda_+},\qquad x^\uparrow:=x+\frac{g(x)}{\Lambda_-},
\end{equation*}
we have
\begin{equation}\label{eq:phi-grad}
\nabla_x\Phi_+(x,y)=x^\downarrow-y^\downarrow,\qquad
\nabla_x\Phi_-(x,z)=x^\uparrow-z^\uparrow,
\end{equation}
and thus \eqref{eq:dai} is equivalently written as
\begin{equation*}
\rho_+\Phi_+(x,y)+\rho_-\Phi_-(x,z)\ge\langle x^\downarrow-y^\downarrow,\;x^\uparrow-z^\uparrow\rangle.
\end{equation*}

Our first main result states that uniform nondegeneracy implies DAI under explicit conditions on the scales and weights.

\begin{theorem}\label{thm:dai}
Every $E\in\classF_{\mu_\pm,L_\pm}(\R^d)$ satisfies $\DAI(\rho_\pm;\Lambda_\pm)$ whenever
\begin{equation}\label{eq:dai-conditions}
\begin{gathered}
\Lambda_\pm\ge L_\pm,\qquad
\rho_\pm\ge\frac{\bar\Lambda-\delta_\Lambda}{\Lambda_\mp},\qquad
\rho_+\rho_-\ge q_\Lambda,
\end{gathered}
\end{equation}
where
\begin{equation*}
c_\Lambda:=\frac{\Lambda_+-\Lambda_-}{2},\qquad
\bar\Lambda:=\frac{\Lambda_++\Lambda_-}{2},\qquad
\delta_\Lambda:=\dist(c_\Lambda,\mathcal{S}),
\end{equation*}
and
\begin{equation}\label{eq:q}
q_\Lambda:=\max_{\lambda\in\mathcal{S}}
\left(1-\frac{\lambda}{\Lambda_+}\right)\left(1+\frac{\lambda}{\Lambda_-}\right)=\frac{\bar\Lambda^2-\delta_\Lambda^2}{\Lambda_+\Lambda_-}.
\end{equation}
\end{theorem}

Conversely, DAI also characterizes uniform nondegeneracy.

\begin{theorem}\label{thm:characterization}
For $L_\pm\ge\mu_\pm>0$, if $E\in C^1(\R^d)$ satisfies
\begin{equation}\label{eq:dai-class}
\DAI\left(\frac{\bar L}{L_\mp};L_\pm\right)
\qquad\text{and}\qquad
\DAI\left(\frac{\sqrt{w(w+\mu_++\mu_-)}}{w+\mu_\mp};w+\mu_\pm\right),
\end{equation}
where $w:=\max\{L_+-\mu_+,L_--\mu_-\}$, then $E\in\classF_{\mu_\pm,L_\pm}(\R^d)$.
\end{theorem}

Using DAI, we establish a global, dimension-free first-order complexity bound for \ADA{} over an explicit range of fixed step sizes (\Cref{thm:ada}).

\subsection{Related Work}

In optimization with mixed-sign curvature, descent--ascent structures are common and typically tied to an explicit curvature splitting: strongly convex--strongly concave minimax problems prescribe fixed minimizing and maximizing variables~\cite{rockafellar1976monotone,lin2020nearoptimal}, whereas gentlest-ascent dynamics~\cite{e2011gentlest}, dimer methods~\cite{henkelman1999dimer}, and high-index saddle dynamics~\cite{yin2019high,su2025improved} reconstruct descent--ascent directions by estimating the Hessian's unstable eigenspace. Shugart and Altschuler~\cite{shugart2025negative} study positive and negative full-gradient steps over successive iterations for quadratic objectives, but their nonlinear convergence analysis retains a prescribed convex--concave splitting and uses a modified randomized schedule. In contrast, \ADA{}, as a deterministic scheme, requires neither a prescribed nor a reconstructed curvature splitting.

Another connection is with difference-of-convex (DC) programming, where an objective is represented as $f_1-f_2$ with $f_1$ and $f_2$ convex, and the standard DC algorithm linearizes $f_2$ and minimizes the resulting convex surrogate~\cite{pham1997convex,lethi2018dc}. As shown in \Cref{sec:variational}, the potential underlying \ADA{} yields an auxiliary DC objective of this form for which the resulting DC iteration recovers the \ADA{} updates exactly. Existing convergence analyses establish linear rates under the classical PL condition~\cite{abbaszadehpeivasti2024dca} and under DC-specific PL conditions formulated through Bregman divergences~\cite{faust2023bregman}. In our setting, the analysis takes uniform nondegeneracy of $E$ itself as its starting point, with DAI turning this geometry into quantitative inequalities for the \ADA{} iterates.

The stationarity condition $g(x)=0$ can also be viewed directly as a symmetric nonlinear equation. Quasi-Newton and derivative-free methods for symmetric nonlinear equations achieve global convergence, with asymptotic superlinear or linear rates under uniform nonsingularity and related regularity assumptions~\cite{li1999globally,zhou2015convergence}. These methods use line searches to select the step sizes, and their analyses do not provide explicit bounds on the total iterations or oracle calls required to reach a prescribed residual tolerance. By contrast, \ADA{} is a fixed two-gradient iteration: each cycle consists simply of one full-gradient descent step followed by one full-gradient ascent step. Under uniform nondegeneracy, \ADA{} converges globally and linearly from arbitrary initialization, with explicit dimension-free first-order complexity controlled by the curvature parameters and fixed step sizes.

\subsection{Organization}

\Cref{sec:dai} establishes the characterization of uniform nondegeneracy by DAI through comparisons of two-sided envelopes. \Cref{sec:ada} analyzes \ADA{} and establishes its global first-order complexity. \Cref{sec:benchmarks} relates the theory to classical strongly convex, strongly convex--strongly concave, and linearly constrained settings. \Cref{sec:conclusion} concludes with a discussion of future directions.

\section{Descent--Ascent Inequalities}\label{sec:dai}

In this section, we establish the equivalence between uniform nondegeneracy and DAI. We begin with some geometric properties of uniformly nondegenerate functions.

\begin{proposition}\label{prop:geometry}
Every $E\in\classF_{\mu_\pm,L_\pm}(\R^d)$ has the following properties.
\begin{enumerate}
\item The gradient $g:=\nabla E$ is a global homeomorphism, and thus the stationary point $x^\star:=g^{-1}(0)$ exists and is unique.
\item The Legendre transform
$$
E^\sharp(p):=\langle p,g^{-1}(p)\rangle-E(g^{-1}(p))
$$
satisfies $E^\sharp\in\classF_{1/L_\pm,1/\mu_\pm}(\R^d)$ with gradient $\nabla E^\sharp=g^{-1}$, and $(E^\sharp)^\sharp=E$. The
identity $E(x)+E^\sharp(g(x))=\langle g(x),x\rangle$ gives the Bregman reciprocity identity
\begin{equation}\label{eq:bregman-duality}
D_E(x_1,x_2)=D_{E^\sharp}(g(x_2),g(x_1)),\qquad \forall x_1,x_2\in\R^d.
\end{equation}
\item The primal Bregman difference satisfies
\begin{equation}\label{eq:primal-models}
-\frac{L_-}{2}\norm{x_1-x_2}^2\le D_E(x_1,x_2)\le\frac{L_+}{2}\norm{x_1-x_2}^2,\quad\forall x_1,x_2\in\R^d.
\end{equation}
The dual Bregman difference satisfies
\begin{equation}\label{eq:dual-models}
-\frac{1}{2\mu_-}\norm{p_1-p_2}^2\le D_{E^\sharp}(p_1,p_2)\le\frac{1}{2\mu_+}\norm{p_1-p_2}^2,\quad\forall p_1,p_2\in\R^d.
\end{equation}
\item The inner secant \eqref{eq:secant-inner} admits the refinement
\begin{equation}\label{eq:secant-refined}
\norm{\Delta g-c_\mu\Delta x}^2-\bar\mu^2\norm{\Delta x}^2
\ge2\bar\mu\abs*{D_E(x_1,x_2)-D_E(x_2,x_1)},\quad\forall x_1,x_2\in\R^d.
\end{equation}
\end{enumerate}
\end{proposition}

\begin{proof}
The first two properties were established in our preceding work~\cite[Section~2.1]{su2026uniform}.
For the quadratic bounds \eqref{eq:primal-models}, Cauchy--Schwarz in \eqref{eq:secant-outer} gives
\begin{equation*}
-\bar{L}\norm{\Delta x}^2\le\langle\Delta g-c_L\Delta x,\Delta x\rangle\le\bar{L}\norm{\Delta x}^2,
\end{equation*}
or equivalently, $-L_-\norm{\Delta x}^2\le\langle\Delta g,\Delta x\rangle\le L_+\norm{\Delta x}^2$.
Integration along the segment from $x_2$ to $x_1$ yields \eqref{eq:primal-models}. Applying the same argument to $E^\sharp\in\classF_{1/L_\pm,1/\mu_\pm}(\R^d)$ gives \eqref{eq:dual-models}.

To prove \eqref{eq:secant-refined}, put $p_i:=g(x_i)$ for $i=1,2$, so $\nabla E^\sharp(p_i)=x_i$. For any trial point $\hat p\in\R^d$, we can write the Bregman difference as
\begin{equation*}
D_{E^\sharp}(p_1,p_2)=\langle\Delta x,p_1-\hat p\rangle-D_{E^\sharp}(\hat p,p_1)+D_{E^\sharp}(\hat p,p_2).
\end{equation*}
Applying the upper bound in \eqref{eq:dual-models} to $D_{E^\sharp}(\hat p,p_2)$ and the lower bound to $D_{E^\sharp}(\hat p,p_1)$, we obtain
\begin{equation*}
D_{E^\sharp}(p_1,p_2)\le\langle\Delta x,p_1-\hat p\rangle+\frac{\norm{\hat p-p_1}^2}{2\mu_-}+\frac{\norm{\hat p-p_2}^2}{2\mu_+}.
\end{equation*}
Minimizing the right-hand side over $\hat p$ gives
\begin{equation*}
D_{E^\sharp}(p_1,p_2)\le\frac{\norm{\Delta g}^2+2\mu_-\langle\Delta g,\Delta x\rangle-\mu_+\mu_-\norm{\Delta x}^2}{2(\mu_++\mu_-)}.
\end{equation*}
Similarly, reverse the bounds to obtain
\begin{equation*}
D_{E^\sharp}(p_1,p_2)\ge\langle\Delta x,p_1-\hat p\rangle-\frac{\norm{\hat p-p_1}^2}{2\mu_+}-\frac{\norm{\hat p-p_2}^2}{2\mu_-},
\end{equation*}
and maximizing the right-hand side over $\hat p$ gives
\begin{equation*}
D_{E^\sharp}(p_1,p_2)\ge\frac{-\norm{\Delta g}^2+2\mu_+\langle\Delta g,\Delta x\rangle+\mu_+\mu_-\norm{\Delta x}^2}{2(\mu_++\mu_-)}.
\end{equation*}
Applying these bounds with the arguments exchanged and subtracting gives
\begin{equation*}
\begin{aligned}
\abs*{D_{E^\sharp}(p_1,p_2)-D_{E^\sharp}(p_2,p_1)}&\le\frac{\norm{\Delta g}^2-2c_\mu\langle\Delta g,\Delta x\rangle-\mu_+\mu_-\norm{\Delta x}^2}{2\bar\mu}\\
&=\frac{\norm{\Delta g-c_\mu\Delta x}^2-\bar\mu^2\norm{\Delta x}^2}{2\bar\mu},
\end{aligned}
\end{equation*}
which becomes \eqref{eq:secant-refined} by the Bregman reciprocity \eqref{eq:bregman-duality}.
\end{proof}

In particular, \eqref{eq:primal-models} ensures that, whenever $\Lambda_\pm\ge L_\pm$, $\Phi_+(\cdot,y)$ and $\Phi_-(\cdot,z)$ are convex and attain their minimum value zero at $y$ and $z$, respectively.

\subsection{Two-Sided Envelopes}

Before turning to the proof of \Cref{thm:dai}, we establish some preparatory properties of two-sided Moreau envelopes.

For $E\in\classF_{\mu_\pm,L_\pm}(\R^d)$, denote $E^\star:=E(x^\star)$ and define the envelopes
\begin{equation}\label{eq:env}
\M_s(x):=\sup_{\xi\in\R^d}
\left\{s(E(\xi)-E^\star)-\frac12\norm{\xi-x}^2\right\},\quad s\in\I:=\left[-L_-^{-1},L_+^{-1}\right].
\end{equation}

When $s\in\I^\circ=(-L_-^{-1},L_+^{-1})$, the objective is strongly concave, and the supremum is finite, attained at the unique maximizer~\cite{su2026uniform}
\begin{equation*}
J_sx:=(\Id-sg)^{-1}(x)=\begin{cases}
\prox_{s(-E)}(x),&s>0,\\
x,&s=0,\\
\prox_{(-s)E}(x),&s<0,
\end{cases}
\end{equation*}
and the envelope satisfies
\begin{equation}\label{eq:env-derivatives}
\nabla_x\M_s(x)=J_sx-x,\qquad
\partial_s\M_s(x)=E(J_sx)-E^\star.
\end{equation}
At the endpoints $s\in\{-L_-^{-1},L_+^{-1}\}$, the envelope may take the value $+\infty$ and remains lower semicontinuous in $x$. For both cases, we have the following proximal identities.

\begin{proposition}\label{prop:proximal}
For every $x\in\R^d$ and $s\in\I$,
\begin{equation}\label{eq:prox-value}
\min_{u\in\R^d}\left\{\M_s(u)+\frac12\norm{u-x}^2\right\}=s(E(x)-E^\star),
\end{equation}
with the unique minimizer $\prox_{\M_s}(x)=x-s g(x)$.
\end{proposition}

\begin{proof}
By \eqref{eq:primal-models}, the function $f_s(\xi):=\frac12\norm{\xi}^2-s(E(\xi)-E^\star)$ is convex and $C^1$. Writing $f_s^*$ for its convex conjugate, \eqref{eq:env} then gives $\M_s(u)=f_s^*(u)-\frac12\norm{u}^2$. The Fenchel--Young inequality~\cite[Sections~2.c--d]{moreau1965proximite} therefore yields
\begin{equation*}
\M_s(u)+\frac12\norm{u-x}^2-s(E(x)-E^\star)
=f_s^*(u)+f_s(x)-\langle u,x\rangle\ge0,
\end{equation*}
with equality precisely when $u\in\partial f_s(x)=\{\nabla f_s(x)\}=\{x-sg(x)\}$.
\end{proof}

We next compare the envelope with the original energy at the same point.

\begin{proposition}\label{prop:envelope-comparison}
For every $x\in\R^d$ and $s\in\I^\circ$,
\begin{equation}\label{eq:env-bound}
(1-s\mu_+)(1+s\mu_-)\M_s(x)\ge s(E(x)-E^\star)+\frac{s^2\mu_+\mu_-}{2}\norm{x-x^\star}^2.
\end{equation}
\end{proposition}

\begin{proof}
For $s\in\I^\circ\setminus\{0\}$, evaluating \eqref{eq:env} at its maximizer $\tilde\xi:=J_sx$ gives
\begin{equation*}
\M_s(x)=s(E(\tilde\xi)-E^\star)-\frac12\norm{\tilde\xi-x}^2=s(E(\tilde{\xi})-E^\star)-\frac{s^2}{2}\norm{g(\tilde\xi)}^2,
\end{equation*}
and hence, by \eqref{eq:env-derivatives},
\begin{equation*}
\partial_s\left[\frac{\M_s(x)}{s}\right]=\frac{\partial_s\M_s(x)}{s}-\frac{\M_s(x)}{s^2}=\frac{E(\tilde\xi)-E^\star}{s}-\frac{\M_s(x)}{s^2}=\frac12\norm{g(\tilde\xi)}^2.
\end{equation*}
Therefore,
\begin{align*}
&\partial_s\left[(1-s\mu_+)(1+s\mu_-)\frac{\M_s(x)}{s}\right]\\
=&\,(1-s\mu_+)(1+s\mu_-)\cdot\frac12\norm{g(\tilde\xi)}^2-(\mu_+-\mu_-+2s\mu_+\mu_-)\cdot\frac{\M_s(x)}{s}\\
=&\,\frac{1+s^2\mu_+\mu_-}{2}\norm{g(\tilde\xi)}^2-2(c_\mu+s\mu_+\mu_-)\left(E(\tilde\xi)-E^\star\right)\\
=&\,\frac{1+s^2\mu_+\mu_-}{2}\norm{g(\tilde\xi)}^2-(c_\mu+s\mu_+\mu_-)\left(D_E(\tilde\xi,x^\star)-D_E(x^\star,\tilde\xi)+\langle g(\tilde\xi),\tilde\xi-x^\star\rangle\right),
\end{align*}
where the last line uses $D_E(\tilde\xi,x^\star)-D_E(x^\star,\tilde\xi)
=2(E(\tilde\xi)-E^\star)-\langle g(\tilde\xi),\tilde\xi-x^\star\rangle$.
Since $x=\tilde\xi-sg(\tilde\xi)$ and $\bar\mu^2-c_\mu^2=\mu_+\mu_-$, we obtain
\begin{equation}\label{eq:env-monotone}
\begin{aligned}
&\partial_s\left[(1-s\mu_+)(1+s\mu_-)\frac{\M_s(x)}{s}\right]-\frac{\mu_+\mu_-}{2}\norm{x-x^\star}^2\\
=&\,\frac12\left[\norm{g(\tilde\xi)-c_\mu(\tilde\xi-x^\star)}^2-\bar\mu^2\norm{\tilde\xi-x^\star}^2\right]\\
&\quad-(c_\mu+s\mu_+\mu_-)\left[D_E(\tilde\xi,x^\star)-D_E(x^\star,\tilde\xi)\right]\ge0,
\end{aligned}
\end{equation}
where the nonnegativity follows from \eqref{eq:secant-refined}, together with the fact that $s\in\I^\circ$ gives $\abs{c_\mu+s\mu_+\mu_-}<\bar\mu$.

At $s=0$, $\M_0(x)=0$, so \eqref{eq:env-derivatives} and $J_0x=x$ give
\begin{equation*}
\lim_{s\to0}(1-s\mu_+)(1+s\mu_-)\frac{\M_s(x)}s=\partial_s\M_s(x)\big|_{s=0}=E(x)-E^\star.
\end{equation*}

Therefore, integrating \eqref{eq:env-monotone} from zero to $s$ gives \eqref{eq:env-bound} for $s\in\I^\circ$.
\end{proof}

This comparison also extends to the energy at a rescaled point.

\begin{corollary}\label{cor:rescaled-anchor}
For every $x\in\R^d$, $s\in\I^\circ$, and $\theta\in[-\mu_-,\mu_+]$,
\begin{equation}\label{eq:env-rescaled}
(1-s\mu_+)(1+s\mu_-)\M_s(x)
\ge s(1-s\theta)^2
\left(E(\tilde{x}_\theta)-E^\star\right)+\frac{s^2\mu_+\mu_-}{2}\norm{x-x^\star}^2,
\end{equation}
where $\tilde{x}_\theta:=x^\star+\frac{x-x^\star}{1-s\theta}$.
\end{corollary}

\begin{proof}
For $\theta\in(-\mu_-,\mu_+)$, the shifted energy $\widetilde E(\xi):=E(\xi)-\frac{\theta}{2}\norm{\xi-x^\star}^2$ belongs to $\classF_{\tilde\mu_\pm,\tilde L_\pm}$ with endpoints $\tilde\mu_\pm:=\mu_\pm\mp\theta$ and $\tilde L_\pm:=L_\pm\mp\theta$, and has the same stationary point and value as $E$.

Since $s\in\I^\circ$, we have $1-s\theta>0$, and thus $\tilde s_\theta:=s/(1-s\theta)$ satisfies
\begin{equation*}
1-\tilde{s}_\theta\tilde{L}_+=\frac{1-sL_+}{1-s\theta}>0,\qquad
1+\tilde{s}_\theta\tilde{L}_-=\frac{1+sL_-}{1-s\theta}>0,
\end{equation*}
or equivalently, $\tilde{s}_\theta\in(-1/\tilde L_-,1/\tilde L_+)$.

Denote the corresponding envelope of $\widetilde E$ by $\widetilde{\M}_{\tilde s_\theta}$, and note that
\begin{equation*}
\begin{aligned}
&\tilde s_\theta\bigl(\widetilde E(\xi)-E^\star\bigr)
-\frac12\norm{\xi-\tilde x_\theta}^2\\
&=\frac1{1-s\theta}
\left[s(E(\xi)-E^\star)-\frac12\norm{\xi-x}^2\right]
-\frac{s\theta}{2(1-s\theta)^2}\norm{x-x^\star}^2,
\end{aligned}
\end{equation*}
so taking the supremum over $\xi\in\R^d$ yields
\begin{equation*}
\widetilde{\M}_{\tilde s_\theta}(\tilde x_\theta)
=\frac{\M_s(x)}{1-s\theta}-\frac{s\theta}{2(1-s\theta)^2}\norm{x-x^\star}^2,
\end{equation*}
which, together with \Cref{prop:envelope-comparison} applied to $\widetilde E$ at $(\tilde s_\theta,\tilde x_\theta)$, gives \eqref{eq:env-rescaled}.

The endpoint cases $\theta=-\mu_-$ or $\theta=\mu_+$ follow by continuity.
\end{proof}

\subsection{Descent--Ascent Inequalities}
With these envelope properties, we can now prove \Cref{thm:dai}. We first compare the two envelopes at independent anchors around the stationary point $x^\star$.

\begin{lemma}\label{lem:two-anchor}
For positive weights $\rho_\pm$ and scales $\Lambda_\pm$ satisfying \eqref{eq:dai-conditions},
\begin{equation}\label{eq:env-pair}
\rho_+\left[\frac12\norm{\xi_+-x^\star}^2+\M_+(\xi_+)\right]
+\rho_-\left[\frac12\norm{\xi_--x^\star}^2+\M_-(\xi_-)\right]\ge\langle\xi_+-x^\star,\xi_--x^\star\rangle
\end{equation}
for every $\xi_+,\xi_-\in\R^d$, where $\M_+:=\M_{1/\Lambda_+}$ and $\M_-:=\M_{-1/\Lambda_-}$.
\end{lemma}
\begin{proof}
First suppose $q_\Lambda>0$. Since $\M_s(x)\ge-\frac12\norm{x-x^\star}^2$, both brackets in \eqref{eq:env-pair} are nonnegative, so it suffices to prove the equality case $\rho_+\rho_-=q_\Lambda$ after reducing $\rho_\pm$.

By definition \eqref{eq:env}, testing $\M_+(\xi_+)$ at
\begin{equation}\label{eq:env-trial}
\tilde\xi:=x^\star+\frac{\xi_--x^\star}{\rho_+}
\end{equation}
gives
\begin{equation*}
\M_+(\xi_+)\ge\frac{E(\tilde\xi)-E^\star}{\Lambda_+}-\frac12\norm{\tilde\xi-\xi_+}^2.
\end{equation*}
Since $\rho_+(\tilde\xi-x^\star)=\xi_--x^\star$ and
\begin{equation*}
\frac12\norm{\xi_+-x^\star}^2-\frac12\norm{\tilde\xi-\xi_+}^2=\langle\tilde\xi-x^\star,\xi_+-x^\star\rangle-\frac12\norm{\tilde\xi-x^\star}^2,
\end{equation*}
we have
\begin{equation*}
\rho_+\left[\frac12\norm{\xi_+-x^\star}^2+\M_+(\xi_+)\right]\ge\langle\xi_+-x^\star,\xi_--x^\star\rangle-\rho_+\left[\frac12\norm{\tilde\xi-x^\star}^2-\frac{E(\tilde\xi)-E^\star}{\Lambda_+}\right].
\end{equation*}
Therefore, it remains to prove
\begin{equation}\label{eq:env-minus}
\rho_-\left[\frac12\norm{\xi_--x^\star}^2+\M_-(\xi_-)\right]
\ge\rho_+\left[\frac12\norm{\tilde\xi-x^\star}^2-\frac{E(\tilde\xi)-E^\star}{\Lambda_+}\right].
\end{equation}

When $\delta_\Lambda=0$, the admissible values from \eqref{eq:dai-conditions} and \eqref{eq:q} force $\rho_+=\bar\Lambda/\Lambda_-$ and $\rho_-=\bar\Lambda/\Lambda_+$. Testing $\M_-(\xi_-)$ at $\tilde\xi$ from \eqref{eq:env-trial} gives
\begin{equation*}
\Lambda_-\left[\frac12\norm{\xi_--x^\star}^2+\M_-(\xi_-)\right]
\ge-(E(\tilde\xi)-E^\star)+\frac{\Lambda_+}{2}\norm{\tilde\xi-x^\star}^2,
\end{equation*}
which is equivalent to \eqref{eq:env-minus}.

When $\delta_\Lambda>0$, the centered energy
\begin{equation}\label{eq:energy-shift}
\widetilde E(\xi):=E(\xi)-\frac{c_\Lambda}{2}\norm{\xi-x^\star}^2
\end{equation}
satisfies, by the triangle and reverse triangle inequalities in \eqref{eq:secants},
\begin{equation*}
\delta_\Lambda\norm{\Delta x}
\le\norm{\nabla\widetilde E(x_1)-\nabla\widetilde E(x_2)}
\le\bar\Lambda\norm{\Delta x}.
\end{equation*}
Since $0<\delta_\Lambda<\bar\Lambda$, $\widetilde E$ therefore belongs to $\classF_{\tilde\mu_\pm,\tilde L_\pm}$ with $\tilde\mu_+=\tilde\mu_-=\delta_\Lambda$ and $\tilde L_+=\tilde L_-=\bar\Lambda$, while its stationary point and value remain $x^\star$ and $E^\star$.

Take $\theta:=\Lambda_-\rho_+-\bar\Lambda$. By \eqref{eq:dai-conditions} and \eqref{eq:q}, the equality $\rho_+\rho_-=q_\Lambda$ gives
\begin{equation*}
\frac{\bar{\Lambda}-\delta_\Lambda}{\Lambda_-}\le\rho_+\le\frac{\bar{\Lambda}+\delta_\Lambda}{\Lambda_-},
\end{equation*}
so $\abs{\theta}\le\delta_\Lambda$.
With $\widetilde{\M}$ denoting the envelope of $\widetilde E$, we apply \Cref{cor:rescaled-anchor} to $\widetilde E$ at $x:=x^\star+\frac{\Lambda_-}{\bar\Lambda}(\xi_--x^\star)$. As $s\downarrow-1/\bar\Lambda$, the rescaled point tends to $\tilde\xi$ in \eqref{eq:env-trial}, which, together with the convexity and lower semicontinuity of $\widetilde{\M}$ in $s$, gives
\begin{equation}\label{eq:env-centered-bound}
\left(1-\frac{\delta_\Lambda^2}{\bar\Lambda^2}\right)\widetilde{\M}_{-1/\bar\Lambda}(x)\ge-\frac{\Lambda_-^2\rho_+^2}{\bar\Lambda^3}\left(\widetilde E\bigl(\tilde{\xi}\bigr)-E^\star\right)+\frac{\delta_\Lambda^2}{2\bar\Lambda^2}\norm{x-x^\star}^2.
\end{equation}
Expanding the defining supremum expresses this envelope in terms of $\M_-$:
\begin{equation}\label{eq:env-centered}
\widetilde{\M}_{-1/\bar\Lambda}(x)
=\frac{\Lambda_-}{\bar\Lambda}\M_-(\xi_-)+\frac{c_\Lambda\Lambda_-}{2\bar\Lambda^2}\norm{\xi_--x^\star}^2,
\end{equation}
and substituting \eqref{eq:energy-shift} and \eqref{eq:env-centered} into \eqref{eq:env-centered-bound} gives
\begin{equation*}
(\bar\Lambda^2-\delta_\Lambda^2)\M_-(\xi_-)+\frac{c_\Lambda^2-\delta_\Lambda^2}{2}\norm{\xi_--x^\star}^2\ge-\Lambda_-\rho_+^2(E(\tilde\xi)-E^\star),
\end{equation*}
which is equivalent to \eqref{eq:env-minus} after simplification.

Thus \eqref{eq:env-minus} holds in both cases, completing the proof when $q_\Lambda>0$.

For the remaining case $q_\Lambda=0$, replace $\mu_\pm$ by $(1-\varepsilon)\mu_\pm$ so that the perturbed $q_\Lambda$ is positive. Applying the preceding result and letting $\varepsilon\downarrow0$ yields \eqref{eq:env-pair}.
\end{proof}

With this lemma, \Cref{thm:dai} follows by applying it to the relative energy.

\begin{proof}[Proof of \Cref{thm:dai}]
First suppose $\rho_\pm>0$.

Fix $x,y,z\in\R^d$ and set $\widehat E:=D_E(\cdot,x)$. Its gradient $\nabla\widehat E=g-g(x)$ has the same secant differences as $g$, so $\widehat E\in\classF_{\mu_\pm,L_\pm}$ with stationary point $x$ and stationary value zero. Write $\widehat{\M}_\pm$ for its envelopes at $1/\Lambda_+$ and $-1/\Lambda_-$, respectively, and take
\begin{equation*}
\xi_+:=x-\nabla_x\Phi_+(x,y)=y-\frac{\nabla\widehat E(y)}{\Lambda_+},\qquad
\xi_-:=x-\nabla_x\Phi_-(x,z)=z+\frac{\nabla\widehat E(z)}{\Lambda_-},
\end{equation*}
which are the unique minimizers given by \Cref{prop:proximal} for $\widehat E$ at $y$ and $z$, respectively. The minimum values \eqref{eq:prox-value} therefore give
\begin{equation*}
\widehat{\M}_+(\xi_+)=\frac{D_E(y,x)}{\Lambda_+}-\frac12\norm{y-\xi_+}^2,\qquad
\widehat{\M}_-(\xi_-)=-\frac{D_E(z,x)}{\Lambda_-}-\frac12\norm{z-\xi_-}^2.
\end{equation*}
Together with $D_E(a,b)+D_E(b,a)=\langle g(a)-g(b),a-b\rangle$, these identities give
\begin{equation*}
\Phi_+(x,y)=\frac12\norm{\xi_+-x}^2+\widehat{\M}_+(\xi_+),\qquad
\Phi_-(x,z)=\frac12\norm{\xi_--x}^2+\widehat{\M}_-(\xi_-).
\end{equation*}
Applying \Cref{lem:two-anchor} to $\widehat E$ at $(\xi_+,\xi_-)$ then establishes \eqref{eq:dai} for positive weights.

If either weight is zero, applying this result with perturbed weights $\rho_\pm+\varepsilon$ and letting $\varepsilon\downarrow0$ yields \eqref{eq:dai}, completing the proof.
\end{proof}

Having established the descent--ascent inequalities on the uniformly nondegenerate function class, we now turn to their converse, showing that they also imply the secant bounds \eqref{eq:secants}.

\begin{theorem}
For $E\in C^1(\R^d)$ satisfying $\DAI(\rho_\pm;\Lambda_\pm)$ with $\rho_+/\Lambda_+=\rho_-/\Lambda_-$,
\begin{equation}\label{eq:secant-balanced}
0\le\left\langle\Delta x-\frac{\Delta g}{\Lambda_+},\;\Delta x+\frac{\Delta g}{\Lambda_-}\right\rangle\le\rho_+\rho_-\norm{\Delta x}^2,\qquad\forall x_1,x_2\in\R^d.
\end{equation}
\end{theorem}

\begin{proof}
We first prove the left inequality in \eqref{eq:secant-balanced}.

For positive weights $\rho_\pm>0$, taking $z=x$ and $y=x$ in \eqref{eq:dai}, respectively, gives $\Phi_+(x,y)\ge0$ and $\Phi_-(x,z)\ge0$, or equivalently,
\begin{equation}\label{eq:dai-models}
-\frac{\Lambda_-}{2}\norm{u-v}^2\le D_E(u,v)\le\frac{\Lambda_+}{2}\norm{u-v}^2,\qquad \forall u,v\in\R^d.
\end{equation}
Therefore, for any $x_1,x_2\in\R^d$ and $\xi\in\R^d$, we have
\begin{align*}
E(\xi)-E(x_1)-\langle g(x_1),\xi-x_1\rangle&\ge-\frac{\Lambda_-}{2}\norm{\xi-x_1}^2,\\
E(\xi)-E(x_2)-\langle g(x_2),\xi-x_2\rangle&\le\frac{\Lambda_+}{2}\norm{\xi-x_2}^2.
\end{align*}
Eliminating $E(\xi)$ gives $D_E(x_1,x_2)\le\frac{\Lambda_-}{2}\norm{\xi-x_1}^2+\frac{\Lambda_+}{2}\norm{\xi-x_2}^2-\langle\Delta g,\xi-x_1\rangle$, and minimizing this right-hand side over $\xi$ gives
\begin{equation*}
D_E(x_1,x_2)\le\frac{\Lambda_+}{2}\norm{\Delta x}^2-\frac{\norm{\Delta g-\Lambda_+\Delta x}^2}{2(\Lambda_++\Lambda_-)}.
\end{equation*}
The same inequality holds for $D_E(x_2,x_1)$ after exchanging the arguments, and adding them gives
\begin{equation*}
D_E(x_1,x_2)+D_E(x_2,x_1)\le\Lambda_+\norm{\Delta x}^2-\frac{\norm{\Delta g-\Lambda_+\Delta x}^2}{\Lambda_++\Lambda_-}.
\end{equation*}
Together with $\langle\Delta g,\Delta x\rangle=D_E(x_1,x_2)+D_E(x_2,x_1)$, this simplifies to
\begin{equation}\label{eq:dai-outer}
0\le\left\langle\Lambda_+\Delta x-\Delta g,\,\Lambda_-\Delta x+\Delta g\right\rangle=\bar\Lambda^2\norm{\Delta x}^2-\norm{\Delta g-c_\Lambda\Delta x}^2,
\end{equation}
which proves the left inequality in \eqref{eq:secant-balanced} when both weights are positive.

For the zero-weight case $\rho_+=\rho_-=0$, introduce the centered energy
\begin{equation}\label{eq:f}
f(\xi):=E(\xi)-\frac{c_\Lambda}{2}\norm{\xi}^2.
\end{equation}
Taking $y=z$ in \eqref{eq:dai} gives $\norm{\nabla f(x)-\nabla f(y)}\ge\bar\Lambda\norm{x-y}$, which makes $\nabla f$ injective with closed image. Since invariance of domain also makes the image open, $\nabla f$ is a global homeomorphism with a $1/\bar\Lambda$-Lipschitz inverse.

For any $x_1,x_2\in\R^d$, we can therefore choose $z$ such that $\nabla f(z)=2\nabla f(x_1)-\nabla f(x_2)$. Adding \eqref{eq:dai} at $(x_1,x_2,z)$ and $(x_1,z,x_2)$ gives
\begin{equation*}
\norm{\nabla f(x_1)-\nabla f(x_2)}^2\le\bar\Lambda^2\langle x_1-x_2,z-x_1\rangle,
\end{equation*}
while the inverse Lipschitz bound gives
\begin{equation*}
\norm{z-x_1}\le\norm{\nabla f(z)-\nabla f(x_1)}/\bar\Lambda=\norm{\nabla f(x_1)-\nabla f(x_2)}/\bar\Lambda.
\end{equation*}
Combining these inequalities by Cauchy--Schwarz implies that $\nabla f$ is $\bar\Lambda$-Lipschitz continuous, which is equivalent to \eqref{eq:dai-outer}.

Consequently, we have established \eqref{eq:dai-outer}, and thus the left inequality in \eqref{eq:secant-balanced}, for all admissible weights $\rho_\pm\ge0$. We now turn to the right inequality.

If $\rho_+\rho_-\Lambda_+\Lambda_-\ge\bar\Lambda^2$, then it follows directly from \eqref{eq:dai-outer}. We therefore consider $\rho_+\rho_-\Lambda_+\Lambda_-<\bar\Lambda^2$, and set
\begin{equation*}
\rho:=\frac{\rho_+\Lambda_-}{\bar\Lambda}=\frac{\rho_-\Lambda_+}{\bar\Lambda}\in[0,1).
\end{equation*}
With the centered energy $f$ in \eqref{eq:f}, the identity $\rho^2\bar\Lambda^2=\rho_+\rho_-\Lambda_+\Lambda_-$ and \eqref{eq:dai-outer} show that the right inequality in \eqref{eq:secant-balanced} reduces to
\begin{equation}\label{eq:f-gap}
\norm{\nabla f(x_1)-\nabla f(x_2)}\overset{?}{\ge}\bar\Lambda\sqrt{1-\rho^2}\,\norm{x_1-x_2},\qquad\forall x_1,x_2\in\R^d.
\end{equation}
To prove this, introduce $\Psi_\pm(x_1,x_2):=\frac12\norm{x_1-x_2}^2\mp\frac{D_f(x_1,x_2)}{\bar\Lambda}=\frac{\Lambda_\pm}{\bar\Lambda}\Phi_\pm(x_1,x_2)$, so that the DAI condition \eqref{eq:dai} becomes
\begin{equation}\label{eq:f-dai}
\rho\bigl[\Psi_+(x,y)+\Psi_-(x,z)\bigr]\ge\left\langle\nabla_x\Psi_+(x,y),\nabla_x\Psi_-(x,z)\right\rangle,
\end{equation}
which means $f$ satisfies $\DAI(\rho;\bar\Lambda)$. Taking $y=z$ in \eqref{eq:f-dai} cancels the Bregman differences and gives
\begin{equation*}
\norm{\nabla f(x_1)-\nabla f(x_2)}\ge\bar\Lambda\sqrt{1-\rho}\,\norm{\Delta x},\qquad\forall x_1,x_2\in\R^d.
\end{equation*}
Together with the $\bar\Lambda$-Lipschitz bound \eqref{eq:dai-outer}, this gives, by \Cref{prop:geometry}, the Legendre dual $f^\sharp$ with $\nabla f^\sharp=(\nabla f)^{-1}$ and
\begin{equation*}
\frac{1}{\bar\Lambda}\le K:=\Lip(\nabla f^\sharp)\le\frac{1}{\bar\Lambda\sqrt{1-\rho}}.
\end{equation*}
The target \eqref{eq:f-gap} therefore becomes
\begin{equation}\label{eq:dual-lip}
K\overset{?}{\le}\frac{1}{\bar\Lambda\sqrt{1-\rho^2}}.
\end{equation}

To prove \eqref{eq:dual-lip}, we first characterize the Lipschitz constant by
\begin{equation}\label{eq:dual-sup}
K=\sup_{p_1\ne p_2}\frac{\abs{\langle\nabla f^\sharp(p_1)-\nabla f^\sharp(p_2),p_1-p_2\rangle}}{\norm{p_1-p_2}^2}.
\end{equation}
Indeed, Cauchy--Schwarz bounds the supremum by $K$. Conversely, let $C$ be any upper bound for the supremum. Integration along the segment from $p_2$ to $p_1$ then yields
\begin{equation*}
\abs*{D_{f^\sharp}(p_1,p_2)}\le\frac C2\norm{p_1-p_2}^2,\qquad\forall p_1,p_2\in\R^d.
\end{equation*}
The same argument that derives \eqref{eq:dai-outer} from \eqref{eq:dai-models}, applied to $f^\sharp$, yields
\begin{equation*}
\norm{\nabla f^\sharp(p_1)-\nabla f^\sharp(p_2)}\le C\norm{p_1-p_2},
\end{equation*}
and thus $K\le C$ for every such $C$, which proves \eqref{eq:dual-sup}.

We now approach the supremum in \eqref{eq:dual-sup} approximately. For every $0<\varepsilon<K$, choose $p_1,p_2\in\R^d$ with $\Delta p:=p_1-p_2\ne0$ such that $\abs*{\langle\nabla f^\sharp(p_1)-\nabla f^\sharp(p_2),\Delta p\rangle}\ge(K-\varepsilon)\norm{\Delta p}^2$. With $\sigma$ denoting the sign, we have
\begin{equation}\label{eq:dual-near-max}
\sigma\langle\nabla f^\sharp(p_1)-\nabla f^\sharp(p_2),\Delta p\rangle\ge(K-\varepsilon)\norm{\Delta p}^2.
\end{equation}

For $\alpha\in[0,1]$, using $K$-Lipschitz continuity between $p_1$ and $p_2+\alpha\Delta p$, together with \eqref{eq:dual-near-max}, gives
\begin{equation}\label{eq:dual-slope-lower}
\begin{aligned}
&\sigma\langle\nabla f^\sharp(p_2+\alpha\Delta p)-\nabla f^\sharp(p_2),\Delta p\rangle\\
=\;&\sigma\langle\nabla f^\sharp(p_1)-\nabla f^\sharp(p_2),\Delta p\rangle-\sigma\langle \nabla f^\sharp(p_1)-\nabla f^\sharp(p_2+\alpha\Delta p),\Delta p\rangle\\
\ge\;&(K-\varepsilon)\norm{\Delta p}^2-K(1-\alpha)\norm{\Delta p}^2=(K\alpha-\varepsilon)\norm{\Delta p}^2.
\end{aligned}
\end{equation}
Also, by $K$-Lipschitz continuity between $p_2$ and $p_2+\alpha\Delta p$,
\begin{equation}\label{eq:dual-segment}
\norm{\nabla f^\sharp(p_2+\alpha\Delta p)-\nabla f^\sharp(p_2)}^2\le K^2\alpha^2\norm{\Delta p}^2.
\end{equation}
Therefore, \eqref{eq:dual-slope-lower} and \eqref{eq:dual-segment} give
\begin{equation}\label{eq:dual-grad-error}
\begin{aligned}
&\norm{\nabla f^\sharp(p_2+\alpha\Delta p)-\nabla f^\sharp(p_2)-\sigma K\alpha\Delta p}^2\\
\le\;&2K^2\alpha^2\norm{\Delta p}^2-2K\alpha(K\alpha-\varepsilon)\norm{\Delta p}^2=2K\alpha\varepsilon\norm{\Delta p}^2,
\end{aligned}
\end{equation}
and thus $\nabla f^\sharp(p_2+\alpha\Delta p)-\nabla f^\sharp(p_2)=\sigma K\alpha\Delta p+O(\sqrt{\varepsilon})\norm{\Delta p}$.
Integrating along the segment gives the estimate uniformly for $\alpha\in[0,1]$:
\begin{equation}\label{eq:dual-bregman-error}
D_{f^\sharp}(p_2+\alpha\Delta p,p_2)=\frac{1}{2}\sigma K\alpha^2\norm{\Delta p}^2+O(\sqrt{\varepsilon})\norm{\Delta p}^2.
\end{equation}

For $\beta_+,\beta_-\in[0,1]$, apply the DAI condition \eqref{eq:f-dai} at the inverse images of three points on this segment:
\begin{equation*}
x:=\nabla f^\sharp(p_2),\qquad y:=\nabla f^\sharp(p_2+\beta_+\Delta p),\qquad z:=\nabla f^\sharp(p_2+\beta_-\Delta p).
\end{equation*}
Their gradient differences are $\nabla f(y)-\nabla f(x)=\beta_+\Delta p$ and $\nabla f(z)-\nabla f(x)=\beta_-\Delta p$, respectively, while Bregman reciprocity \eqref{eq:bregman-duality} gives $D_f(x,y)=D_{f^\sharp}(p_2+\beta_+\Delta p,p_2)$ and the corresponding identity for $z$. Substituting these relations into \eqref{eq:f-dai} and using \eqref{eq:dual-grad-error} and \eqref{eq:dual-bregman-error} gives
\begin{equation*}
\frac{\rho}{2}\left[\left(K^2-\frac{\sigma K}{\bar\Lambda}\right)\beta_+^2+\left(K^2+\frac{\sigma K}{\bar\Lambda}\right)\beta_-^2\right]\ge\left(K^2-\frac{1}{\bar\Lambda^2}\right)\beta_+\beta_--O(\sqrt\varepsilon).
\end{equation*}
With $\beta_\pm=\sqrt{\frac12\pm\frac{\sigma}{2K\bar\Lambda}}$, this reduces to
\begin{equation*}
\frac{\rho}{2}\left(K^2-\frac{1}{\bar\Lambda^2}\right)\ge\frac{\left(K^2-\bar\Lambda^{-2}\right)^{3/2}}{2K}-O(\sqrt\varepsilon).
\end{equation*}
Passing to the limit $\varepsilon\downarrow0$, we conclude that
\begin{equation*}
\rho K\ge\sqrt{K^2-\frac{1}{\bar\Lambda^2}},\quad\text{or equivalently, }\quad K\le\frac{1}{\bar\Lambda\sqrt{1-\rho^2}}.
\end{equation*}
This proves \eqref{eq:dual-lip}, hence \eqref{eq:f-gap} and the right inequality in \eqref{eq:secant-balanced}.
\end{proof}

Applying this result to \eqref{eq:dai-class} recovers the original secants.

\begin{proof}[Proof of \Cref{thm:characterization}]
Both conditions in \eqref{eq:dai-class} satisfy weight balance. Applying the left inequality of \eqref{eq:secant-balanced} to the first condition gives
\begin{equation*}
0\le\langle L_+\Delta x-\Delta g,L_-\Delta x+\Delta g\rangle=\bar L^2\norm{\Delta x}^2-\norm{\Delta g-c_L\Delta x}^2,
\end{equation*}
which is exactly \eqref{eq:secant-outer}. For the second DAI condition, $c_\Lambda=c_\mu$, $\bar\Lambda=\bar\mu+w$, and $\rho_+\rho_-\Lambda_+\Lambda_-=w(w+2\bar\mu)$, so the right inequality in \eqref{eq:secant-balanced} gives
\begin{equation*}
(\bar\mu+w)^2\norm{\Delta x}^2-\norm{\Delta g-c_\mu\Delta x}^2\le w(w+2\bar\mu)\norm{\Delta x}^2,
\end{equation*}
which is exactly \eqref{eq:secant-inner}. Thus $E\in\classF_{\mu_\pm,L_\pm}(\R^d)$.
\end{proof}

Together, \Cref{thm:dai,thm:characterization} show that the two DAI conditions in \eqref{eq:dai-class} characterize $\classF_{\mu_\pm,L_\pm}(\R^d)$.

\section{Alternating Descent Ascent}\label{sec:ada}

We now turn to the convergence analysis of \ADA{} (\Cref{alg:ada}), which alternates between gradient descent and gradient ascent.

To gain intuition for why this works, we start by considering the quadratic function $E(x)=\lambda \norm{x}^2/2$, where $\lambda\in\mathcal{S}=[-L_-,-\mu_-]\cup[\mu_+,L_+]$. With prescribed step sizes $\tau_\pm=1/\Lambda_\pm>0$, a pair of descent and ascent steps transforms the iterate $x$ as
\begin{equation*}
x\xRightarrow[]{\;\text{descent}\;}\left(1-\frac{\lambda}{\Lambda_+}\right)x\xRightarrow[]{\;\text{ascent}\;}\left(1+\frac{\lambda}{\Lambda_-}\right)\left(1-\frac{\lambda}{\Lambda_+}\right)x,
\end{equation*}
which also accounts for the definition of $q_\Lambda$ in \eqref{eq:q}. In particular, assuming
\begin{equation}\label{eq:ada-scales}
\Lambda_\pm\ge L_\pm,\qquad -\mu_-<\Lambda_+-\Lambda_-<\mu_+
\end{equation}
ensures that the multiplier lies in $[0,1)$ for all $\lambda\in\mathcal{S}$, and the iterates converge to the stationary point.
For the remainder of this section, we adopt the assumption \eqref{eq:ada-scales} on the scale parameters $\Lambda_\pm$ and turn to the analysis of \ADA{} for general $E\in\classF_{\mu_\pm,L_\pm}(\R^d)$.

\subsection{Variational Structure and Convergence}\label{sec:variational}

\ADA{} admits a simple variational structure that underlies its convergence. By introducing the potential
\begin{equation}\label{eq:v}
V(y,z):=\frac12\norm{y-z}^2+\M_-(y)+\M_+(z),
\end{equation}
with $\M_+:=\M_{\tau_+}$ and $\M_-:=\M_{-\tau_-}$, we can express the descent and ascent steps using \Cref{prop:proximal} as
\begin{equation}\label{eq:v-minimizers}
y^\downarrow=\operatorname*{argmin}_z\left\{\frac12\norm{y-z}^2+\M_+(z)\right\},\;\;
z^\uparrow=\operatorname*{argmin}_y\left\{\frac12\norm{y-z}^2+\M_-(y)\right\},
\end{equation}
and thus the alternating steps can be interpreted as alternating minimization of the potential $V$ over its two coordinates:
\begin{equation*}
V(y_t,z_{t-1})
\xRightarrow[\quad z_t:=y_t^\downarrow\quad]{\min_z}
V(y_t,z_t)
\xRightarrow[\quad y_{t+1}:=z_t^\uparrow\quad]{\min_y}
V(y_{t+1},z_t)
\Longrightarrow\cdots
\end{equation*}

\begin{remark}
The potential $V$ in \eqref{eq:v} also reveals a connection with difference-of-convex (DC) programming~\cite{lethi2018dc}.
Indeed, by \eqref{eq:env},
\begin{equation*}
f_1(y):=\M_-(y)+\frac12\norm{y}^2=\sup_{\xi\in\R^d}\left\{\langle y,\xi\rangle-\frac12\norm{\xi}^2-\tau_-(E(\xi)-E^\star)\right\}
\end{equation*}
is convex, while \eqref{eq:primal-models} implies that $f_2(y):=\frac12\norm{y}^2-\tau_+(E(y)-E^\star)$ is also convex. The potential can then be expressed as
\begin{equation*}
V(y,z)=f_1(y)+f_2^*(z)-\langle y,z\rangle,
\end{equation*}
and minimizing over $z$ yields the reduced DC objective $\min_z V(y,z)=f_1(y)-f_2(y)$. Since $f_1^*(z)=\frac12\norm{z}^2+\tau_-(E(z)-E^\star)$, the standard DC algorithm~\cite{pham1997convex} takes the form
\begin{equation*}
\begin{aligned}
z_t&=\nabla f_2(y_t)=y_t-\tau_+g(y_t),\\
y_{t+1}&=\nabla f_1^*(z_t)=z_t+\tau_-g(z_t),
\end{aligned}
\end{equation*}
which precisely recovers the \ADA{} updates \eqref{eq:v-minimizers}.
\end{remark}

Moreover, the potential satisfies the following variational property.
\begin{proposition}
$V(y,z)\ge0$, with equality if and only if $y=z=x^\star$.
\end{proposition}
\begin{proof}
By admissibility \eqref{eq:ada-scales}, $\delta_\Lambda>\abs{c_\Lambda}$ and $q_\Lambda<1$, so \Cref{lem:two-anchor} with weights $\rho_+=\rho_-=1$ at anchors $\xi_+=z$ and $\xi_-=y$ gives
\begin{equation*}
\frac12\norm{z-x^\star}^2+\M_+(z)+\frac12\norm{y-x^\star}^2+\M_-(y)\ge\langle z-x^\star,y-x^\star\rangle,\qquad\forall y,z\in\R^d,
\end{equation*}
which is exactly $V(y,z)\ge0$.

When $V(y,z)=0$, the pair $(y,z)$ minimizes $V$, so \eqref{eq:v-minimizers} gives $z=y^\downarrow$ and $y=z^\uparrow$. Consequently, $g(y)-g(z)=(\Lambda_+-\Lambda_-)(y-z)$, and \eqref{eq:secant-inner} implies
\begin{equation*}
0\ge\bigl(\bar\mu-\abs{\Lambda_+-\Lambda_--c_\mu}\bigr)\norm{y-z}.
\end{equation*}
By \eqref{eq:ada-scales}, the coefficient is positive, so $y=z$, yielding $g(y)=\Lambda_+(y-z)=0$, and hence $y=z=x^\star$.
Conversely, \Cref{prop:proximal} at $x=x^\star$ gives $\M_\pm(x^\star)=0$ and hence $V(x^\star,x^\star)=0$.
\end{proof}

To quantify the decrease of $V$ under the coordinate updates, we show the following estimates.

\begin{lemma}\label{lem:single-step}
For every $x,y,z\in\R^d$,
\begin{equation}\label{eq:phi-bound}
\Phi_+(x,y)\ge\frac{\norm{\nabla_x\Phi_+(x,y)}^2}{2(1+\tau_+/\tau_-)},\qquad \Phi_-(x,z)\ge\frac{\norm{\nabla_x\Phi_-(x,z)}^2}{2(1+\tau_-/\tau_+)}.
\end{equation}
\end{lemma}
\begin{proof}
With $y$ fixed, \eqref{eq:primal-models} implies that the Bregman difference of $\Phi_+(\cdot,y)$ satisfies
\begin{equation*}
D_{\Phi_+(\cdot,y)}(\xi,x)=\frac12\norm{\xi-x}^2-\tau_+D_E(\xi,x)\le\frac{1+\tau_+/\tau_-}{2}\norm{\xi-x}^2.
\end{equation*}
Since $\Phi_+(\xi,y)\ge0$, expanding this Bregman difference gives
\begin{equation*}
\Phi_+(x,y)\ge-\langle\nabla_x\Phi_+(x,y),\xi-x\rangle-\frac{1+\tau_+/\tau_-}{2}\norm{\xi-x}^2,\qquad\forall\xi\in\R^d.
\end{equation*}
Choosing $\xi=x-\nabla_x\Phi_+(x,y)/(1+\tau_+/\tau_-)$ gives the first inequality in \eqref{eq:phi-bound}; replacing $E$ by $-E$ and exchanging the scales gives the second likewise.
\end{proof}

In the following discussion, we adopt the convention of starting \ADA{} with the larger step size. Specifically, if $\tau_+\ge\tau_-$, we start with a descent step, and if $\tau_+<\tau_-$, we start with an ascent step.
Therefore, without loss of generality, replacing $E$ by $-E$ and exchanging the $+$ and $-$ parameters if necessary, we may assume $\tau_+\ge\tau_-$ and write the descent-first \ADA{} trajectory as
\begin{equation}\label{eq:ada}
z_t=y_t-\tau_+g(y_t),\qquad y_{t+1}=z_t+\tau_-g(z_t),\qquad y_0=x_0,\qquad t=0,1,\ldots.
\end{equation}
The alternating minimization structure then gives the monotone decrease:
\begin{equation*}
V(y_1,z_0)\ge V(y_1,z_1)\ge V(y_2,z_1)\ge V(y_2,z_2)\ge\cdots\ge0,
\end{equation*}
while \Cref{prop:proximal} implies
\begin{align*}
V(y_1,z_0)&=\M_+(z_0)+\min_{y\in\R^d}\left\{\frac12\norm{y-z_0}^2+\M_-(y)\right\}=\M_+(z_0)-\tau_-(E(z_0)-E^\star)\\
&=\tau_+(E(y_0)-E^\star)-\tau_-(E(z_0)-E^\star)-\frac12\norm{y_0-z_0}^2<\infty.
\end{align*}
Moreover, the consecutive decrease of $V$ is bounded below by
\begin{equation*}
\begin{aligned}
V(y_t,z_t)-V(y_{t+1},z_t)&=\frac12\norm{y_t-z_t}^2-\frac12\norm{y_t-z_{t-1}}^2+\tau_-\bigl(E(z_t)-E(z_{t-1})\bigr)\\
&=\Phi_-(z_t,z_{t-1})\ge\frac{\norm{y_{t+1}-y_t}^2}{2(1+\tau_-/\tau_+)},
\end{aligned}
\end{equation*}
where the second equality uses $y_t=z_{t-1}+\tau_-g(z_{t-1})$, and the last inequality follows from \eqref{eq:phi-grad} and \Cref{lem:single-step}.
Together with nonnegativity of $V$, this gives $y_{t+1}-y_t\to0$. Combining this with \eqref{eq:ada} and the inner secant \eqref{eq:secant-inner} yields
\begin{equation*}
\bigl(\bar\mu-\abs{\Lambda_+-\Lambda_--c_\mu}\bigr)\norm{y_t-z_t}\le\Lambda_-\norm{y_{t+1}-y_t}\longrightarrow0.
\end{equation*}
The coefficient is positive by \eqref{eq:ada-scales}, so $\norm{y_t-z_t}\to0$, and thus $g(y_t)=\Lambda_+(y_t-z_t)\to0$.

Consequently, $y_t\to x^\star$, and $z_t=y_t^\downarrow\to x^\star$.

\subsection{Oracle Complexity}

We now proceed to quantify the oracle complexity of \ADA{} convergence.

\begin{theorem}\label{thm:ada}
For every fixed admissible scale pair $\Lambda_\pm$, \ADA{} starting from any $x_0\in\R^d$ finds an $\varepsilon$-stationary point for $0<\varepsilon<\norm{g(x_0)}$ with oracle complexity
\begin{equation}\label{eq:ada-complexity}
Q=O\left(\kappa_\Lambda\left\lceil\log\frac{\norm{g(x_0)}}{\varepsilon}\right\rceil\right),\qquad\kappa_\Lambda:=\frac{1}{1-q_\Lambda}.
\end{equation}
\end{theorem}

\begin{proof}
Consider the corrected energy
\begin{equation}\label{eq:w}
W_t:=\tau_+(E(y_t)-E^\star)
+\frac{\tau_-}{\tau_+}(\kappa_\Lambda-1)\Phi_+(z_t,y_t)
+\frac12\norm{y_{t+1}-z_t}^2.
\end{equation}
We first show that $W_t\ge0$. By the definition of $\Phi_+$, we have the equivalent expression
\begin{equation}\label{eq:w-at-z}
W_t=\tau_+(E(z_t)-E^\star)+\left[1+\frac{\tau_-}{\tau_+}(\kappa_\Lambda-1)\right]\Phi_+(z_t,y_t)
+\frac12\norm{y_t-z_t}^2+\frac12\norm{y_{t+1}-z_t}^2.
\end{equation}
Thus it suffices to bound the first two terms.

By \eqref{eq:secants}, $E\in\classF_{\widehat\mu_\pm,\Lambda_\pm}$ with $\widehat\mu_\pm:=\delta_\Lambda\pm c_\Lambda\le\mu_\pm$ and
\begin{equation*}
\Lambda_+-\widehat\mu_+=\Lambda_--\widehat\mu_-,
\qquad
\kappa_\Lambda\tau_+\tau_-\widehat\mu_+\widehat\mu_-=1.
\end{equation*}
Applying \eqref{eq:dual-models} to this class, together with $D_{E^\sharp}(0,g(x))=D_E(x,x^\star)=E(x)-E^\star$ by \eqref{eq:bregman-duality}, yields
\begin{equation}\label{eq:energy-residual}
-\frac{\norm{g(x)}^2}{2\widehat\mu_-}
\le E(x)-E^\star
\le\frac{\norm{g(x)}^2}{2\widehat\mu_+},\qquad \forall x\in\R^d.
\end{equation}
In particular, at $z_t$,
\begin{equation}\label{eq:energy-lower}
E(z_t)-E^\star\ge-\frac{\norm{g(z_t)}^2}{2\widehat\mu_-}=-\frac{1}{2\tau_-^2\widehat\mu_-}\norm{y_{t+1}-z_t}^2.
\end{equation}
Meanwhile, by \eqref{eq:phi-bound},
\begin{equation}\label{eq:phi-lower}
\Phi_+(z_t,y_t)\ge\frac{\norm{\nabla_x\Phi_+(z_t,y_t)}^2}{2(1+\tau_+/\tau_-)}
=\frac{\tau_+^2}{2\tau_-(\tau_++\tau_-)}\norm{y_{t+1}-z_t}^2.
\end{equation}
Substituting \eqref{eq:energy-lower} and \eqref{eq:phi-lower} into \eqref{eq:w-at-z} gives
\begin{equation}\label{eq:w-lower}
W_t\ge\frac12\norm{y_t-z_t}^2
+\frac{\tau_-+\tau_+\kappa_\Lambda(1-\tau_+\widehat\mu_+)^2}
{2(\tau_++\tau_-)}\norm{y_{t+1}-z_t}^2
\ge0.
\end{equation}

We may also bound $W_t$ in \eqref{eq:w} from above. Applying \eqref{eq:energy-residual} to $y_t$ gives
\begin{equation}\label{eq:energy-upper}
E(y_t)-E^\star\le\frac{\norm{g(y_t)}^2}{2\widehat\mu_+}=\frac{1}{2\tau_+^2\widehat\mu_+}\norm{y_t-z_t}^2,
\end{equation}
while, by the lower bound in \eqref{eq:primal-models},
\begin{equation}\label{eq:phi-upper}
\Phi_+(z_t,y_t)=\frac12\norm{z_t-y_t}^2-\tau_+D_E(z_t,y_t)
\le\frac{\tau_++\tau_-}{2\tau_-}\norm{y_t-z_t}^2.
\end{equation}
Substituting \eqref{eq:energy-upper}, \eqref{eq:phi-upper}, and \eqref{eq:phi-lower} into \eqref{eq:w} and using $\tau_-\le\tau_+$ gives the upper bound
\begin{equation}\label{eq:w-upper}
2W_t\le(3\kappa_\Lambda+2)\norm{y_t-z_t}^2.
\end{equation}

We now turn to estimating the decrease
\begin{equation}\label{eq:w-diff}
\begin{aligned}
W_t-W_{t+1}=&\,\tau_+\bigl(E(y_t)-E(y_{t+1})\bigr)\\
&+\frac{\tau_-}{\tau_+}(\kappa_\Lambda-1)\bigl[\Phi_+(z_t,y_t)-\Phi_+(z_{t+1},y_{t+1})\bigr]\\
&+\frac12\norm{y_{t+1}-z_t}^2-\frac12\norm{y_{t+2}-z_{t+1}}^2
\end{aligned}
\end{equation}
by bounding each term separately.

The first term is bounded using the upper bound in \eqref{eq:primal-models}:
\begin{equation}\label{eq:energy-change}
\begin{aligned}
&\tau_+\bigl(E(y_{t+1})-E(y_t)\bigr)
\le\langle y_t-z_t,y_{t+1}-y_t\rangle+\frac12\norm{y_{t+1}-y_t}^2\\
=&\,\frac12\norm{y_{t+1}-z_t}^2-\frac12\norm{y_t-z_t}^2=\frac12\norm{y_{t+1}-z_t}^2-\frac{\tau_+^2}{2}\norm{g(y_t)}^2.
\end{aligned}
\end{equation}

The second term, involving the difference of $\Phi_+$, is more subtle.
Using \eqref{eq:ada},
\begin{equation}\label{eq:phi-diff}
\begin{aligned}
\frac{\tau_-}{\tau_+}\bigl[\Phi_+(z_t,y_t)-\Phi_+(z_{t+1},y_{t+1})\bigr]
={}&\frac{\tau_-}{\tau_+}\bigl[\Phi_+(z_t,y_t)+\Phi_+(y_{t+1},z_{t+1})\bigr]\\
&-\langle y_{t+1}-z_{t+1},y_{t+2}-z_{t+1}\rangle.
\end{aligned}
\end{equation}
To bound the sum of $\Phi_+$ on the right-hand side, we apply \Cref{thm:dai} at $(z_t,y_t,y_{t+1})$ and $(y_{t+1},z_{t+1},z_t)$ with weights $(q_\Lambda\tau_-/\tau_+,\tau_+/\tau_-)$, which gives
\begin{align*}
q_\Lambda\frac{\tau_-}{\tau_+}\Phi_+(z_t,y_t)+\frac{\tau_+}{\tau_-}\Phi_-(z_t,y_{t+1})
&\ge\langle y_{t+1}-z_{t+1},y_{t+1}-z_t\rangle,\\
q_\Lambda\frac{\tau_-}{\tau_+}\Phi_+(y_{t+1},z_{t+1})+\frac{\tau_+}{\tau_-}\Phi_-(y_{t+1},z_t)
&\ge\langle y_{t+1}-z_{t+1},y_{t+2}-z_{t+1}\rangle.
\end{align*}
Adding them and using $\frac{\tau_+}{\tau_-}\bigl[\Phi_-(z_t,y_{t+1})+\Phi_-(y_{t+1},z_t)\bigr]
=\langle y_{t+1}-z_{t+1},y_{t+1}-z_t\rangle$, \eqref{eq:phi-diff}, and $\kappa_\Lambda-1=q_\Lambda/(1-q_\Lambda)$ gives
\begin{equation*}
(\kappa_\Lambda-1)\frac{\tau_-}{\tau_+}\bigl[\Phi_+(z_t,y_t)-\Phi_+(z_{t+1},y_{t+1})\bigr]
\ge\langle y_{t+1}-z_{t+1},y_{t+2}-z_{t+1}\rangle.
\end{equation*}
Moreover, applying \eqref{eq:phi-bound} to $\Phi_+$ at $(z_{t+1},y_{t+1})$ and $(y_{t+1},z_{t+1})$ gives
\begin{equation*}
\begin{aligned}
\langle y_{t+1}-z_{t+1},y_{t+2}-z_{t+1}\rangle
&=\frac{\tau_-}{\tau_+}\bigl[\Phi_+(z_{t+1},y_{t+1})+\Phi_+(y_{t+1},z_{t+1})\bigr]\\
&\ge\frac{\tau_+}{\tau_++\tau_-}\norm{y_{t+2}-z_{t+1}}^2,
\end{aligned}
\end{equation*}
and consequently,
\begin{equation}\label{eq:phi-decrease}
(\kappa_\Lambda-1)\frac{\tau_-}{\tau_+}
\bigl[\Phi_+(z_t,y_t)-\Phi_+(z_{t+1},y_{t+1})\bigr]
\ge\frac{\tau_+}{\tau_++\tau_-}\norm{y_{t+2}-z_{t+1}}^2.
\end{equation}
With \eqref{eq:energy-change}, \eqref{eq:phi-decrease}, and $\tau_+\ge\tau_-$, \eqref{eq:w-diff} gives the sufficient decrease estimate
\begin{equation}\label{eq:w-decrease}
W_t-W_{t+1}\ge\frac12\norm{y_t-z_t}^2.
\end{equation}
Combining this with \eqref{eq:w-upper} then gives the contraction $W_{t+1}\le\left(1-\frac{1}{3\kappa_\Lambda+2}\right)W_t$, yielding linear convergence of $W_t$ to zero. Moreover, for every integer $T\ge1$, summing \eqref{eq:w-decrease} from $t$ to $t+T$ and using \eqref{eq:w-lower} and \eqref{eq:w-upper} gives
\begin{equation*}
\min_{t+1\le i\le t+T}\norm{g(y_i)}^2
\le\frac1T\sum_{i=t+1}^{t+T}\norm{g(y_i)}^2\le\frac{2W_t-\norm{y_t-z_t}^2}{T\tau_+^2}
\le\frac{4\kappa_\Lambda}{T}\norm{g(y_t)}^2,
\end{equation*}
which shows that the gradient norm halves within at most $T_{1/2}:=\lceil16\kappa_\Lambda\rceil$ iterations. Repeating this process with $\lceil\log_2(\norm{g(x_0)}/\varepsilon)\rceil$ successive halvings yields an $\varepsilon$-stationary point, proving \eqref{eq:ada-complexity}.
\end{proof}

In particular, choosing $\Lambda_\pm=\mu_\pm+w$, with $w:=\max\{L_+-\mu_+,L_--\mu_-\}$, minimizes $\kappa_\Lambda$ over all scales satisfying \eqref{eq:ada-scales}, yielding
\begin{equation}\label{eq:kappa-optimal}
\min \kappa_\Lambda=\left(1+\frac{w}{\mu_+}\right)\left(1+\frac{w}{\mu_-}\right).
\end{equation}
Indeed, for $\vartheta:=\Lambda_+-\Lambda_-\in(-\mu_-,\mu_+)$, the minimum of $\lambda(\lambda-\vartheta)$ over $\mathcal{S}$ is attained at one of its inner endpoints $-\mu_-$ and $\mu_+$, so \eqref{eq:q} and \eqref{eq:ada-scales} give
\begin{equation}\label{eq:kappa-scales}
\kappa_\Lambda=\frac{\Lambda_+\Lambda_-}{\min\{\mu_+(\mu_+-\vartheta),\mu_-(\mu_-+\vartheta)\}}\ge\frac{\max\{L_+(L_+-\vartheta),L_-(L_-+\vartheta)\}}{\min\{\mu_+(\mu_+-\vartheta),\mu_-(\mu_-+\vartheta)\}}.
\end{equation}
When $\vartheta\le\mu_+-\mu_-$, the denominator in the final ratio of \eqref{eq:kappa-scales} is $\mu_-(\mu_-+\vartheta)$, and each numerator term divided by it is nonincreasing in $\vartheta$. When $\vartheta\ge\mu_+-\mu_-$, the denominator is $\mu_+(\mu_+-\vartheta)$, and both ratios are nondecreasing. The final ratio in \eqref{eq:kappa-scales} is therefore minimized at $\vartheta=\mu_+-\mu_-$, with value $(1+w/\mu_+)(1+w/\mu_-)$. The choice $\Lambda_\pm=\mu_\pm+w$ attains this bound, proving \eqref{eq:kappa-optimal}.

Moreover, since $w\le L_{\max}:=\max\{L_+,L_-\}$, \eqref{eq:kappa-optimal} gives $\min\kappa_\Lambda\le\frac{4L_{\max}^2}{\mu_+\mu_-}$, so \ADA{} sharpens the oracle complexity bound established for paired proximal descent in our preceding work~\cite{su2026uniform}.

We conclude the analysis of \ADA{} with an illustrative example of a spiral-saddle function. Using polar coordinates $(r\cos\varphi,r\sin\varphi)\in\R^2$, consider
\begin{equation}\label{eq:spiral}
E_\alpha(r\cos\varphi,r\sin\varphi):=
\begin{cases}
\frac12r^2\sin(2\varphi-2\alpha\log r)&r>0,\\
0&r=0.
\end{cases}
\end{equation}

When $0<\alpha<\sqrt{(\sqrt{13}-3)/2}$, the function
$E_\alpha\in C^1(\R^2)$ is uniformly nondegenerate and has the origin as its unique stationary point; see \Cref{app:spiral} for the proof and explicit curvature bounds $L(\alpha)>\mu(\alpha)>0$.

Moreover, as $r\downarrow0$, the phase $\psi:=2\varphi-2\alpha\log r$ winds indefinitely, causing the convex and concave directions to keep rotating, and hence no fixed direction is convex or concave in any neighborhood of the origin.

Nevertheless, for $\alpha=0.4$, \Cref{fig:spiral} shows \ADA{} converging
to the origin from three distinct initializations using the same fixed
step sizes $\tau_+=\tau_-=1/L(\alpha)$.

\begin{figure}[ht]
\centering
\includegraphics[width=0.48\textwidth]{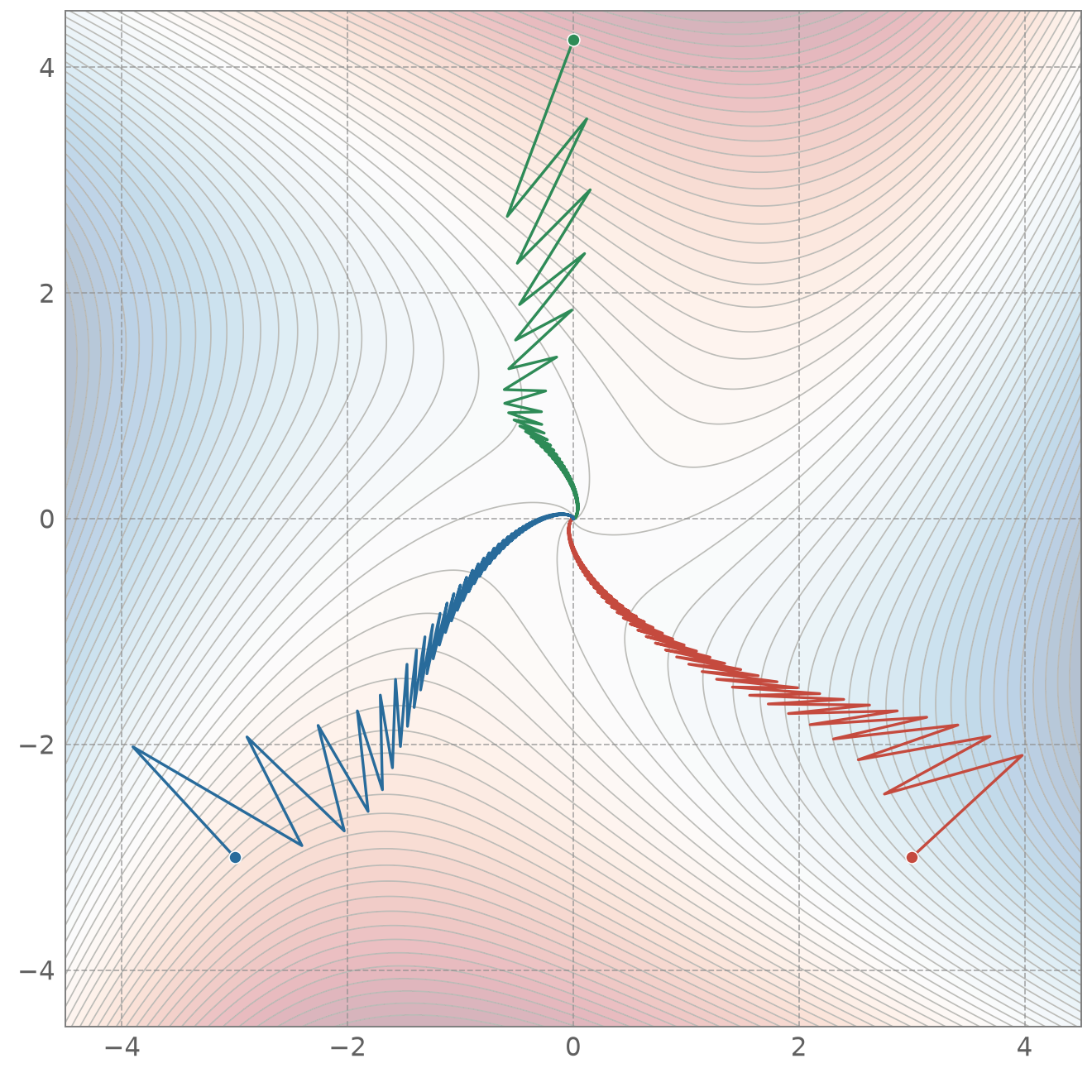}
\caption{\ADA{} trajectories on the spiral saddle $E_\alpha$ with $\alpha=0.4$ from three different initial points.}
\label{fig:spiral}
\end{figure}

\section{Classical Benchmarks}\label{sec:benchmarks}

Having established the theory under uniform nondegeneracy, we now relate it to several classical settings in first-order optimization.

\subsection{Strongly Convex Optimization}

We first show that the classical smooth strongly convex setting is recovered precisely as a limit of the uniform nondegeneracy framework. Specifically, consider the limiting parameter regime with fixed positive-curvature bounds $\mu_+$ and $L_+$ and a negative band that recedes to infinity:
\begin{equation}\label{eq:sc-family}
\mu_+=\mu,\quad L_+=L,\qquad \mu_-\to\infty,\quad L_-=\mu_-+(L-\mu)\to\infty.
\end{equation}

For fixed $x_1,x_2\in\R^d$, the inner secant \eqref{eq:secant-inner} gives
\begin{equation*}
\norm{\Delta g}^2-(\mu_+-\mu_-)\langle\Delta g,\Delta x\rangle-\mu_+\mu_-\norm{\Delta x}^2\ge0.
\end{equation*}
Dividing by $\mu_-$ and taking the limit $\mu_-\to\infty$ gives $\langle\Delta g,\Delta x\rangle\ge\mu\norm{\Delta x}^2$, which is the classical characterization of $\mu$-strong convexity~\cite[Section~2.1]{nesterov2018lectures}. Likewise, the outer secant \eqref{eq:secant-outer} gives $\langle\Delta g,\Delta x\rangle\le L\norm{\Delta x}^2$, characterizing $L$-smoothness.

Furthermore, with $\Lambda_+=L$ and $\Lambda_-=L_-\to\infty$, DAI recovers the classical interpolation inequality, as stated in the following proposition.

\begin{proposition}
When $L>\mu$, the limit of \eqref{eq:dai} with weights $(1,q_\Lambda)$ reduces to the interpolation inequality~\cite{taylor2017interpolation}
\begin{equation*}
D_E(\xi,x)\ge
\frac{\mu}{2}\norm{\xi-x}^2
+\frac{\norm{g(\xi)-g(x)-\mu(\xi-x)}^2}{2(L-\mu)},\quad\forall\xi,x\in\R^d.
\end{equation*}
\end{proposition}
\begin{proof}
Under the limit \eqref{eq:sc-family}, $\Phi_-(x,z)\to\norm{x-z}^2/2$ and $\nabla_x\Phi_-(x,z)\to x-z$, so \eqref{eq:dai} with weights $(1,q_\Lambda)$ tends to
\begin{equation*}
\Phi_+(x,y)+\frac{1-1/\kappa}{2}\norm{x-z}^2\ge\langle\nabla_x\Phi_+(x,y),x-z\rangle,\qquad\kappa:=L/\mu.
\end{equation*}
Optimizing over $z$ at $z=x-\nabla_x\Phi_+(x,y)/(1-1/\kappa)$ gives
\begin{equation*}
D_E(x,y)\le\frac L2\norm{x-y}^2-\frac{\norm{L(x-y)-g(x)+g(y)}^2}{2(L-\mu)}.
\end{equation*}
Setting $y=\xi$, we obtain the desired inequality:
\begin{align*}
D_E(\xi,x)
&=\langle g(\xi)-g(x),\xi-x\rangle-D_E(x,\xi)\\
&\ge\langle g(\xi)-g(x),\xi-x\rangle-\frac L2\norm{\xi-x}^2+\frac{\norm{L(\xi-x)-g(\xi)+g(x)}^2}{2(L-\mu)}\\
&=\frac\mu2\norm{\xi-x}^2+\frac{\norm{g(\xi)-g(x)-\mu(\xi-x)}^2}{2(L-\mu)}.
\end{align*}
\end{proof}

On the algorithmic side, the same limit recovers classical gradient descent. Indeed, since
$\tau_+=1/L$, $\tau_-\to0$ and $\kappa_\Lambda\to\kappa$, \ADA{} reduces exactly to classical gradient descent with fixed step size $1/L$ and complexity $O(\kappa\log(\norm{g(x_0)}/\varepsilon))$.

Thus, the classical smooth strongly convex interpolation inequality and gradient descent arise simultaneously from the same limit within the uniform nondegeneracy framework.

\subsection{Strongly Convex--Strongly Concave Optimization}

Consider the strongly convex--strongly concave minimax problem
\begin{equation*}
\min_{u\in\R^n}\max_{v\in\R^m} E(u,v),
\end{equation*}
whose saddle operator $F(x):=(\nabla_u E,-\nabla_v E)$, with $x=(u,v)$, is $\mu$-strongly monotone and $L$-Lipschitz continuous. Classical gradient descent--ascent (GDA) explicitly uses this minimax splitting, updating the minimizing variable $u$ and the maximizing variable $v$, respectively, by~\cite{facchinei2003variational}
\begin{equation*}
u_{t+1}=u_t-\eta\nabla_u E(u_t,v_t),\qquad v_{t+1}=v_t+\eta\nabla_v E(u_t,v_t).
\end{equation*}

With $\eta=\mu/L^2$, the standard contraction estimate~\cite[Theorem~12.1.2]{facchinei2003variational} implies
\begin{equation*}
\norm{F(x_{t+1})}\le\sqrt{1-\kappa^{-2}}\norm{F(x_t)}, \quad\kappa:=L/\mu,
\end{equation*}
and GDA obtains an $\varepsilon$-stationary point in $O(\kappa^2\log(\norm{F(x_0)}/\varepsilon))$ gradient evaluations.

The same SCSC assumptions already imply uniform nondegeneracy.
Indeed, since $F$ is $\mu$-strongly monotone,
\begin{equation*}
\norm{F(x)-F(x^\prime)}\cdot\norm{x-x^\prime}\ge\langle F(x)-F(x^\prime),x-x^\prime\rangle\ge\mu\norm{x-x^\prime}^2.
\end{equation*}
Together with $L$-Lipschitzness of $F$, and noting that
$F=(\nabla_uE,-\nabla_vE)$ differs from
$g=\nabla E=(\nabla_uE,\nabla_vE)$ only by a sign change in the $v$-component, we obtain
\begin{equation*}
\mu\norm{x-x^\prime}\le\norm{g(x)-g(x^\prime)}\le L\norm{x-x^\prime},
\qquad
\forall x,x^\prime\in\R^{n+m}.
\end{equation*}
Hence $E$ is uniformly nondegenerate with $\mu_+=\mu_-=\mu$ and $L_+=L_-=L$.
Taking $\Lambda_+=\Lambda_-=L$ gives $\kappa_\Lambda=\kappa^2$, and \ADA{} with $\tau_+=\tau_-=1/L$ updates the joint variable $x=(u,v)$ simultaneously by
\begin{equation*}
x_{t+1/2}=x_t-\frac{1}{L}g(x_t),\qquad x_{t+1}=x_{t+1/2}+\frac{1}{L}g(x_{t+1/2}).
\end{equation*}
\Cref{thm:ada} shows that it reaches an $\varepsilon$-stationary point in $O(\kappa^2\log(\norm{g(x_0)}/\varepsilon))$ gradient evaluations.
Thus, the classical GDA complexity order is already attained under uniform nondegeneracy alone, without using the prescribed minimizing--maximizing splitting.

\subsection{Linearly Constrained Strongly Convex Optimization}

Consider
\begin{equation*}
\min_{x\in\R^n} f(x)\quad\text{subject to }Ax=b,
\end{equation*}
where $f\in C^1(\R^n)$ is $\mu$-strongly convex and $L$-smooth, $A\in\R^{m\times n}$ has full row rank, and $b\in\R^m$. For $\gamma>0$, its KKT system~\cite[Section~5.5.3]{boyd2004convex}
is equivalently the stationarity condition for the scaled Lagrangian $\mathcal{L}_\gamma(x,\lambda):=f(x)+\gamma\langle\lambda,Ax-b\rangle$:
\begin{equation*}
\nabla\mathcal{L}_\gamma(x^\star,\lambda^\star)=\Bigl(\nabla f(x^\star)+\gamma A^\top\lambda^\star,\;\gamma(Ax^\star-b)\Bigr)=0,
\end{equation*}
where $\gamma\lambda$ is the usual Lagrange multiplier.
Although $\mathcal{L}_\gamma$ is affine in $\lambda$ and hence is not strongly convex--strongly concave, the full-row-rank coupling through $A$ is sufficient to make it uniformly nondegenerate.

\begin{proposition}\label{prop:kkt}
$\mathcal{L}_\gamma$ is uniformly nondegenerate with endpoints
\begin{equation}\label{eq:kkt-endpoints}
\begin{aligned}
\mu_+&=\mu,
&L_+&=\frac{\sqrt{L^2+4\gamma^2\sigma_{\max}^2}+L}{2},\\
\mu_-&=\frac{\sqrt{L^2+4\gamma^2\sigma_{\min}^2}-L}{2},
&L_-&=\frac{\sqrt{\mu^2+4\gamma^2\sigma_{\max}^2}-\mu}{2},
\end{aligned}
\end{equation}
where $\sigma_{\min}:=\sigma_{\min}(A)$ and $\sigma_{\max}:=\sigma_{\max}(A)$.
\end{proposition}
\begin{proof}
For every $x_1,x_2\in\R^n$, there is a symmetric matrix $H$ such that
\begin{equation*}
\nabla f(x_1)-\nabla f(x_2)=H(x_1-x_2),\qquad \mu \Id\preceq H\preceq L\Id.
\end{equation*}
Therefore, the Lagrangian gradient difference at $(x_1,\lambda_1)$ and $(x_2,\lambda_2)$ is given by
\begin{equation*}
\nabla\mathcal{L}_\gamma(x_1,\lambda_1)-\nabla\mathcal{L}_\gamma(x_2,\lambda_2)=\mathcal{K}(H)\cdot\begin{pmatrix}
x_1-x_2\\\lambda_1-\lambda_2
\end{pmatrix},\qquad
\mathcal{K}(H):=\begin{pmatrix}H&\gamma A^\top\\\gamma A&0\end{pmatrix},
\end{equation*}
and it suffices to verify $\spec\mathcal{K}(H)\subseteq[-L_-,-\mu_-]\cup[\mu_+,L_+]$, with the details left to \Cref{app:kkt}.
\end{proof}

To balance the constraint coupling with the smoothness scale of $f$, choose
\begin{equation}\label{eq:kkt-scales}
\gamma=L/\sigma_{\max}(A),\qquad
\Lambda_+=L_+,\qquad
\Lambda_-=\mu_-+L_+-\mu,
\end{equation}
which gives $\norm{\gamma A}=L$. With step sizes $\tau_\pm=1/\Lambda_\pm$, alternating full-gradient descent and ascent of $\mathcal{L}_\gamma$ in the joint variable $(x,\lambda)$ yields the following complexity bound.

\begin{corollary}
\ADA{} applied to $\mathcal{L}_\gamma$ with the parameters in \eqref{eq:kkt-scales}, starting from any $(x_0,\lambda_0)\in\R^{n+m}$ with the larger step size, finds a pair $(x,\lambda)$ satisfying $\norm{\nabla\mathcal{L}_\gamma(x,\lambda)}\le\varepsilon$ with oracle complexity
\begin{equation}\label{eq:kkt-complexity}
Q=O\left(\kappa_f\kappa_A^2\left\lceil\log\frac{\norm{\nabla\mathcal{L}_\gamma(x_0,\lambda_0)}}{\varepsilon}\right\rceil\right),\quad
\kappa_f:=\frac L\mu,\quad \kappa_A:=\frac{\sigma_{\max}(A)}{\sigma_{\min}(A)}.
\end{equation}
In particular, \ADA{} admits a matrix-free implementation: each evaluation of $\nabla\mathcal{L}_\gamma$ requires only one evaluation of $\nabla f$ and the matrix-vector products $Ax$ and $A^\top\lambda$.
\end{corollary}
\begin{proof}
The endpoints \eqref{eq:kkt-endpoints} satisfy
\begin{equation*}
L_+-L_-=\frac{L+\mu+\sqrt{L^2+4\gamma^2\sigma_{\max}^2}-\sqrt{\mu^2+4\gamma^2\sigma_{\max}^2}}{2}\ge\mu,
\end{equation*}
so \eqref{eq:kkt-scales} gives $\Lambda_-\ge L_-$ and $\Lambda_+-\Lambda_-=\mu-\mu_-\in(-\mu_-,\mu)$, proving admissibility. Since $c_\Lambda=c_\mu$ and $\delta_\Lambda=\bar\mu$, \eqref{eq:q} gives
\begin{equation*}
\kappa_\Lambda=\frac{\Lambda_+\Lambda_-}{\mu\mu_-}
=\frac{L_+}{\mu}\left(1+\frac{L_+-\mu}{\mu_-}\right).
\end{equation*}
At $\gamma=L/\sigma_{\max}$, the endpoint formulas become
\begin{equation*}
L_+=\frac{1+\sqrt5}{2}L,\qquad
\mu_-=\frac{2L}{\kappa_A^2\left(1+\sqrt{1+4/\kappa_A^2}\right)}=\Theta\left(\frac L{\kappa_A^2}\right).
\end{equation*}
Thus $\kappa_\Lambda=O(\kappa_f\kappa_A^2)$, and \Cref{thm:ada} yields the complexity bound \eqref{eq:kkt-complexity}.
\end{proof}

The same dependence on $\kappa_f$ and $\kappa_A$ appears in the convergence rates of classical incremental and simultaneous primal--dual gradient methods~\cite[Theorem~1 and Corollary~1]{alghunaim2020linear}, which likewise admit matrix-free implementations. Their guarantees are stated in terms of weighted primal--dual iterate errors, whereas \eqref{eq:kkt-complexity} controls the scaled KKT residual. Here the curvature gap is generated by strong convexity together with the full-rank constraint coupling, rather than by a prescribed convex--concave splitting. Once this uniform nondegeneracy is established, \ADA{} applies without further specialization to the constrained problem.

\section{Conclusion}\label{sec:conclusion}

Classical first-order optimization is largely built on models in which curvature is both uniformly separated from zero and organized according to prescribed positive and negative orientations. The results of this paper show that these two roles can be separated more sharply than classical theory might suggest. Uniform nondegeneracy retains the curvature gap while allowing its orientation to vary across the domain, yet still admits a coherent first-order descent--ascent structure. DAI captures this interplay quantitatively, characterizes uniform nondegeneracy, and reduces in the strongly convex limit to the classical interpolation inequality.

At the algorithmic level, \ADA{} gives a particularly simple realization of this geometry: it alternates fixed full-gradient descent and ascent steps without prescribing or reconstructing curvature directions. The variational formulation shows that these two updates are exact alternating minimization steps for a common nonnegative potential, which vanishes only at $(x^\star,x^\star)$. Together with this variational structure, DAI yields global linear convergence from arbitrary initialization with a dimension-free first-order complexity bound of $O(\kappa_\Lambda\log(\norm{g(x_0)}/\varepsilon))$.

The classical benchmarks and the spiral-saddle example show that the same fixed-step descent--ascent scheme applies across very different curvature geometries whenever uniform nondegeneracy holds. The underlying curvature gap may arise from one-sided strong convexity, from strongly convex--strongly concave structure with a prescribed minimizing--maximizing splitting, or from full-rank constraint coupling, while the curvature orientation may even rotate indefinitely near the stationary point. These cases point to a common message: for baseline global first-order solvability, the essential information is the uniform separation of curvature from zero rather than the particular mechanism that produces the gap or prescribes its orientation.

These benchmarks also suggest a natural direction for future research: acceleration toward a square-root dependence on $\kappa_\Lambda$. In the strongly convex limit, $\kappa_\Lambda\to\kappa$, so a dependence on $\sqrt{\kappa_\Lambda}$ would recover the classical accelerated scale $O(\sqrt{\kappa}\log(1/\varepsilon))$~\cite{nesterov2018lectures}. In the SCSC setting, $\kappa_\Lambda=\kappa^2$, and the same improvement would give the $O(\kappa\log(1/\varepsilon))$ scale attained by extragradient methods~\cite{mokhtari2020unified}. For linearly constrained strongly convex optimization, our parameter choice gives $\kappa_\Lambda=O(\kappa_f\kappa_A^2)$, so the corresponding square-root dependence becomes $O(\sqrt{\kappa_f}\,\kappa_A\log(1/\varepsilon))$, consistent with the accelerated primal--dual scale for this problem class~\cite{salim2022optimal}.
These correspondences suggest that the linear dependence on $\kappa_\Lambda$ in the present complexity bound may not be intrinsic to uniform nondegeneracy. The resulting question is whether uniform nondegeneracy alone admits a globally convergent first-order method with complexity
\begin{equation*}
O\left(
\sqrt{\kappa_\Lambda}
\log\frac{\norm{g(x_0)}}{\varepsilon}
\right).
\end{equation*}

An affirmative answer would show that acceleration, like baseline global solvability, does not require a prescribed curvature orientation. Conversely, a lower bound separating the general uniformly nondegenerate class from its classical specializations would identify which additional structure acceleration genuinely requires. DAI provides a natural starting point for this question by capturing descent and ascent simultaneously at the first-order level.

\appendix

\section{Uniform Nondegeneracy of the Spiral Example}\label{app:spiral}
In this appendix, we verify that, whenever $0<\alpha<\sqrt{(\sqrt{13}-3)/2}$, the function $E_\alpha$ in \eqref{eq:spiral} belongs to $\classF_{\mu_\pm,L_\pm}(\R^2)$ with $\mu_+=\mu_-=\mu(\alpha)$ and $L_+=L_-=L(\alpha)$ given by
\begin{equation*}
\mu(\alpha):=\frac{1-\sqrt{2\alpha^2+4\alpha^4+\alpha^6}}{1+4\alpha+2\alpha^2}>0, \quad L(\alpha):=1+4\alpha+2\alpha^2.
\end{equation*}

We first show that $E_\alpha\in C^1(\R^2)$, for which the only nontrivial point is the origin.
Since $\abs{E_\alpha}\le r^2/2$, the extension at the origin is differentiable, with $\nabla E_\alpha(0)=0$.

For $r>0$, with $\psi:=2\varphi-2\alpha\log r$ denoting the phase, differentiating $E_\alpha$ in the orthonormal polar frame $e_r=(\cos\varphi,\sin\varphi)$, $e_\varphi=(-\sin\varphi,\cos\varphi)$ gives
\begin{equation*}
\nabla E_\alpha=r(\sin\psi-\alpha\cos\psi)e_r+r\cos\psi\,e_\varphi,
\end{equation*}
which tends to zero as $r\downarrow0$, and thus $E_\alpha\in C^1(\R^2)$.

It remains to establish the uniform secant bounds \eqref{eq:secants}. Consider the smooth regularization $E_{\alpha,\delta}\in C^\infty(\R^2)$ with parameter $\delta>0$ defined by
\begin{equation*}
E_{\alpha,\delta}(r\cos\varphi,r\sin\varphi)=\frac12r^2\sin\psi_\delta,\qquad
\psi_\delta:=2\varphi-\alpha\log(r^2+\delta^2).
\end{equation*}
With $\beta:=r^2/(r^2+\delta^2)\in(0,1)$, $\fc:=\cos\psi_\delta$, and $\fs:=\sin\psi_\delta$, we have
\begin{equation}\label{eq:spiral-grad-delta}
\nabla E_{\alpha,\delta}=r(\fs-\alpha\beta \fc)e_r+r\fc e_\varphi.
\end{equation}
Thus $E_{\alpha,\delta}$ converges to $E_\alpha$ in $C^1_{\text{loc}}(\R^2)$ as $\delta\downarrow0$, and hence it suffices to establish $E_{\alpha,\delta}\in\classF_{\mu(\alpha),L(\alpha)}(\R^2)$ for all $\delta>0$.

Since $E_{\alpha,\delta}$ is smooth, this is equivalent to the uniform spectral bounds \eqref{eq:spectrum} on its Hessian. Differentiating \eqref{eq:spiral-grad-delta} in the polar frame gives
\begin{equation}\label{eq:spiral-hessian}
H_{\alpha,\delta}=\begin{pmatrix}
(1-2\alpha^2\beta^2)\fs-\alpha\beta(5-2\beta)\fc&\fc+2\alpha\beta \fs\\
\fc+2\alpha\beta \fs&-\fs-\alpha\beta \fc
\end{pmatrix},
\end{equation}
or equivalently,
\begin{equation*}
\begin{aligned}
H_{\alpha,\delta}={}&\begin{pmatrix}\fs&\fc\\\fc&-\fs\end{pmatrix}-\alpha\beta(3-\beta)\fc\Id\\
&+\alpha\beta\begin{pmatrix}-(2-\beta)\fc&2\fs\\2\fs&(2-\beta)\fc\end{pmatrix}-2\alpha^2\beta^2\fs\begin{pmatrix}1&0\\0&0\end{pmatrix}.
\end{aligned}
\end{equation*}
Since $\fc^2+\fs^2=1$ and $0<\beta<1$, the triangle inequality gives
\begin{equation*}
\norm{H_{\alpha,\delta}}\le1+\alpha\beta(3-\beta)\abs{\fc}+\alpha\beta\sqrt{(2-\beta)^2\fc^2+4\fs^2}+2\alpha^2\beta^2\abs{\fs}\le1+4\alpha+2\alpha^2,
\end{equation*}
which proves the outer spectral bound $L(\alpha)$.

Consequently, for the inner bound, it suffices to show that $-\det H_{\alpha,\delta}\ge\mu(\alpha)L(\alpha)$. Expanding the determinant in \eqref{eq:spiral-hessian} gives
\begin{equation}\label{eq:spiral-det}
-\det H_{\alpha,\delta}=\begin{pmatrix}\fc&\fs\end{pmatrix}K_{\alpha,\beta}\begin{pmatrix}\fc\\\fs\end{pmatrix},\;
K_{\alpha,\beta}:=\begin{pmatrix}
1-\alpha^2\beta^2(5-2\beta)&\alpha\beta^2(1-\alpha^2\beta)\\
\alpha\beta^2(1-\alpha^2\beta)&1+2\alpha^2\beta^2
\end{pmatrix}.
\end{equation}

We first show that $K_{\alpha,\beta}$ is positive definite.
Direct computation implies that $\operatorname{tr}K_{\alpha,\beta}=2-\alpha^2\beta^2(3-2\beta)\in(0,2)$ and
\begin{equation*}
\det K_{\alpha,\beta}=1-\beta^2(3-2\beta+\beta^2)\cdot\alpha^2-2\beta^4(5-3\beta)\cdot\alpha^4-\beta^6\cdot\alpha^6.
\end{equation*}
Since $\beta\in(0,1)$, all coefficients of the powers of $\alpha$ are nonincreasing in $\beta$, so
\begin{equation*}
\det K_{\alpha,\beta}\ge1-2\alpha^2-4\alpha^4-\alpha^6=(1+\alpha^2)(1-3\alpha^2-\alpha^4)>0,
\end{equation*}
and thus $K_{\alpha,\beta}$ is positive definite.
Writing $0<\lambda_1\le\lambda_2$ for its eigenvalues, we obtain from the trace and determinant estimates
\begin{equation*}
1-2\alpha^2-4\alpha^4-\alpha^6\le\det K_{\alpha,\beta}=\lambda_1\lambda_2\le\lambda_1(2-\lambda_1)=1-(1-\lambda_1)^2,
\end{equation*}
which, together with \eqref{eq:spiral-det}, gives
\begin{equation*}
-\det H_{\alpha,\delta}\ge\lambda_1\ge1-\sqrt{2\alpha^2+4\alpha^4+\alpha^6}=\mu(\alpha)L(\alpha),
\end{equation*}
concluding the proof that $E_{\alpha,\delta}\in\classF_{\mu(\alpha),L(\alpha)}(\R^2)$ for all $\delta>0$. Taking the limit $\delta\downarrow0$ then gives $E_\alpha\in\classF_{\mu(\alpha),L(\alpha)}(\R^2)$, as desired.

\section{Spectral Bounds for Saddle-Point Matrices}\label{app:kkt}
In this appendix, we verify the spectral bounds in \Cref{prop:kkt} by comparing $\mathcal{K}(H)$ with the matrices obtained from the two scalar endpoints. Since $\mu\Id\preceq H\preceq L\Id$,
\begin{equation}\label{eq:kkt-order}
\mathcal{K}(\mu\Id)\preceq\mathcal{K}(H)\preceq\mathcal{K}(L\Id).
\end{equation}

It therefore suffices to compute the spectra of $\mathcal{K}(\lambda\Id)$ for $\lambda\in\{\mu,L\}$ and compare the ordered eigenvalues.
Take a full singular value decomposition $A=U\begin{pmatrix}\Sigma&0\end{pmatrix}V^\top$ with $\Sigma=\diag\bigl(\sigma_1(A),\ldots,\sigma_m(A)\bigr)$, where $U\in\R^{m\times m}$ and $V\in\R^{n\times n}$ are orthogonal. In these coordinates,
\begin{equation}\label{eq:kkt-svd}
\begin{pmatrix}V^\top&0\\0&U^\top\end{pmatrix}
\mathcal{K}(\lambda\Id)
\begin{pmatrix}V&0\\0&U\end{pmatrix}
=
\begin{pmatrix}
\lambda\Id_m&0&\gamma\Sigma\\
0&\lambda\Id_{n-m}&0\\
\gamma\Sigma&0&0
\end{pmatrix}.
\end{equation}
Pairing each singular direction with its corresponding multiplier coordinate in \eqref{eq:kkt-svd} gives the $m$ blocks
\begin{equation*}
\begin{pmatrix}\lambda&\gamma\sigma_i(A)\\\gamma\sigma_i(A)&0\end{pmatrix},\qquad i=1,\ldots,m,
\end{equation*}
while the remaining $n-m$ directions contribute the eigenvalue $\lambda$. Each coupled block has characteristic equation $r^2-\lambda r-\gamma^2\sigma_i(A)^2=0$, and hence eigenvalues
\begin{equation}\label{eq:kkt-roots}
\frac{\lambda-\sqrt{\lambda^2+4\gamma^2\sigma_i(A)^2}}{2}<0,\qquad
\frac{\lambda+\sqrt{\lambda^2+4\gamma^2\sigma_i(A)^2}}{2}>0.
\end{equation}
Thus $\mathcal{K}(\lambda\Id)$ has exactly $m$ negative and $n$ positive eigenvalues.

Writing $\lambda_j(M)$ for the $j$th smallest eigenvalue of a symmetric matrix $M$, the min--max principle~\cite{horn2013matrix} applied to \eqref{eq:kkt-order} gives $\lambda_j\bigl(\mathcal{K}(\mu\Id)\bigr)
\le\lambda_j\bigl(\mathcal{K}(H)\bigr)
\le\lambda_j\bigl(\mathcal{K}(L\Id)\bigr)$ for $j=1,\ldots,n+m$.

Since the negative root in \eqref{eq:kkt-roots} decreases with the singular value, whereas the positive root increases, the endpoints in \eqref{eq:kkt-endpoints} satisfy
\begin{equation*}
\begin{aligned}
\lambda_1\bigl(\mathcal{K}(\mu\Id)\bigr)
&=\frac{\mu-\sqrt{\mu^2+4\gamma^2\sigma_{\max}^2}}{2}=-L_-,\\
\lambda_m\bigl(\mathcal{K}(L\Id)\bigr)
&=\frac{L-\sqrt{L^2+4\gamma^2\sigma_{\min}^2}}{2}=-\mu_-,\\
\lambda_{m+1}\bigl(\mathcal{K}(\mu\Id)\bigr)&\ge\mu=\mu_+,\\
\lambda_{n+m}\bigl(\mathcal{K}(L\Id)\bigr)
&=\frac{L+\sqrt{L^2+4\gamma^2\sigma_{\max}^2}}{2}=L_+,
\end{aligned}
\end{equation*}
which places the first $m$ eigenvalues of $\mathcal{K}(H)$ in $[-L_-,-\mu_-]$ and the remaining $n$ in $[\mu_+,L_+]$, proving $\spec\mathcal{K}(H)\subseteq[-L_-,-\mu_-]\cup[\mu_+,L_+]$.

\section*{Acknowledgments}
The authors used OpenAI Codex (GPT-6 Astra) during the research and preparation of this manuscript. All mathematical statements, derivations, proofs, algorithms, and numerical results included in the manuscript were independently examined and validated by the authors. The authors take full responsibility for all content.

\bibliographystyle{siamplain}
\bibliography{references}

\end{document}